\documentclass{amsart}
\usepackage{amsmath,amssymb,amsthm,tikz-cd,hyperref,stmaryrd,tensor}
\usepackage[math]{cellspace}
\usepackage[symbol]{footmisc}
\usepackage{perpage}
\MakePerPage{footnote}

\usepackage{mathtools}
\usepackage[inline]{enumitem}
\usepackage[mathscr]{eucal}

\newcommand{\character}{\operatorname{char}}

\newcommand{\rank}{\operatorname{rk}}

\newcommand{\redu}[1]{\overline {#1}}

\newcommand{\St}[2]{\operatorname{St}_{#2}({#1})}
\newcommand{\Lk}[2]{\operatorname{Lk}_{#2}({#1})}
\newcommand{\susp}[1]{\Sigma {#1}}
\newcommand{\cone}[1]{C {#1}}
\newcommand{\im}{\operatorname{im}}
\newcommand{\sk}{\operatorname{sk}}

\newcommand{\Aut}{\operatorname{Aut}}

\newcommand{\N}{\mathbb Z_{\geq 1}}
\newcommand{\Z}{\mathbb Z}
\newcommand{\R}{\mathbb R}
\newcommand{\Q}{\mathbb Q}
\newcommand{\C}{\mathbb C}
\newcommand{\F}{\mathbb F}
\renewcommand{\k}{\mathbf k}

\newcommand{\hocolim}{\operatorname{hocolim}}

\newcommand{\Quad}[1]{\operatorname{Quad}(#1 )}

\DeclareFontEncoding{LS1}{}{}
\DeclareFontSubstitution{LS1}{stix}{m}{n}
\DeclareSymbolFont{symbols2}{LS1}{stixfrak} {m} {n}
\DeclareMathSymbol{\orthosum}{\mathbin}{symbols2}{"A8}

\renewcommand{\H}{\mathbb H}
\newcommand{\E}{\mathbb E}

\newtheorem{thm}{Theorem}[section]
\newtheorem{lem}[thm]{Lemma}
\newtheorem{prop}[thm]{Proposition}
\newtheorem{cor}[thm]{Corollary}

\newtheorem{claim}{Claim}

\newcommand{\thistheoremname}{}
\newtheorem{genericthm}[thm]{\thistheoremname}
\newenvironment{namedtheorem}[1]
  {\renewcommand{\thistheoremname}{#1}%
   \begin{genericthm}}
   {\end{genericthm}}

\usepackage{etoolbox}
\AtEndEnvironment{proof}{\setcounter{claim}{0}}
\theoremstyle{definition}

\newtheorem{theorem}{Theorem}

\newtheorem{corollary}[theorem]{Corollary}

\newtheorem{exmp}[thm]{Example}
\newtheorem*{defn}{Definition}

\theoremstyle{remark}
\newtheorem*{rmk}{Remark}

\newenvironment{claimproof}{\par\noindent\textit{Proof of claim.}}{\footnotesize{\hfill$\blacksquare$}\newline}

\title[Twisted homology stability of $O_n$ for valuation rings]{Twisted homology stability of $O_n$\\for valuation rings}
\author{Oscar Harr}
\date{\today}

\numberwithin{equation}{section}

\begin{document}
\maketitle
\begin{abstract}
  In this article, we extend an argument of Vogtmann in order
  to show homology stability of the Euclidean orthogonal group
  $O_n(A)$ when $A$ is a
  valuation ring subject to arithmetic conditions
  on either its residue or its quotient field.
  In particular, it is shown that if $A$ is a henselian
  valuation ring, then the groups $O_n(A)$
  exhibit homology stability if the residue field of $A$
  has finite Pythagoras number.
  Our results include those of Vogtmann,
  and hold with various twisted coefficients.
  Using these results, we give analogues for fields $F\neq\R$ of
  some computations that appear in the study of
  scissor congruences.
\end{abstract}

\setcounter{section}{-1}
\section{Introduction}
Let $R$ be a commutative ring.
For each $n\in\N$, consider the quadratic modules
\begin{align*}
  \E_R^n &= (R^n,n\langle 1\rangle) \tag{Euclidean $n$-space}, \\
  \H_R^{2n} &= (R^{2n},n\langle 1,-1\rangle) \tag{hyperbolic $2n$-space}.
\end{align*}
(We leave out subscripts when no ambiguity may arise.)
Here $n\langle 1\rangle$ denotes the Euclidean quadratic form
$\sum_{i=1}^n X_i^2$ and $n\langle 1,-1\rangle$ is the hyperbolic quadratic form
$\sum_{i=1}^{2n} (-1)^{i+1}X_i^2$.
We write
\begin{align*}
  O_n(R) &= O(n\langle 1\rangle )\tag{orthogonal group},\\
  O_{n,n}(R) &= O(n\langle 1,-1\rangle ) \tag{split-orthogonal group}
\end{align*}
for the isometry groups of these modules.\footnote{
  In hermitian $K$-theory,
  the split-orthogonal group is often referred to simply as the orthogonal group,
  and occasionally it is even denoted $O_n(R)$.
  We follow conventions that agree with classical definitions,
  so for instance $O_n(\R )$ is the usual orthogonal group
  (viewed as a discrete group).
}
The homology of split-orthogonal groups has been intensely
studied due to its importance in hermitian $K$-theory (recalled in Section~\ref{subsec:k-thy}),
and it is known to stabilize for a large class of
rings~\cite{vogtmann81,betley87,charney87,mirzaii-kallen2001}.

Less is known about the homology of Euclidean orthogonal groups,
which has only been shown 
to stabilize over fields with
certain arithmetic properties~\cite{vogtmann82};
over the rings of $S$-integers in number fields for certain sets of
places $S$~\cite{collinet11};
and for finite rings in which two is invertible~\cite{wang-kannan22}. 
Some work has been done on improving Vogtmann's stability range for
$O_n(F)$ when $F$ is a field subjected to stricter arithmetic
conditions, e.g. by Cathelineau~\cite{cathelineau2007} for
infinite Pythagorean fields
and Sprehn--Wahl~\cite{sprehn-wahl} for fields whose Stufe is
less than or equal to two.

\medskip

In this article, we extend Vogtmann's methods to valuation rings.
Our proof fits into the stability framework of
Randal-Williams--Wahl~\cite{RWW17},
and thus we also get homology stability with various twisted coefficients.
Before we state our theorem, 
we recall some algebraic notions:
\begin{itemize}
\item The \emph{Pythagoras number} $P(R)$
 of a ring $R$ is the smallest $p\in\N$
so that every sum of squares in $R$ is a sum
of $p$ squares, if any such number $p$ exists;
if no such number exists, put $P(R) = \infty$.
\item A \emph{valuation ring} is an integral domain $A$
with fraction field $K$, so that for each $x\in K$ either
$x\in A$ or $x^{-1}\in A$.
Equivalently, there exists a totally ordered abelian group $\Gamma$
and a valuation $\nu\colon K\to\Gamma\cup\lbrace\infty\rbrace$
so that $A = \lbrace x\in K\mid\nu (x)\geq 0\rbrace$.
\item A local ring $A$ with residue field $\k$ is
  \emph{henselian} if it satisfies Hensel's lemma:
for any polynomial $f\in A\lbrack X\rbrack$,
if $\alpha\in\k$ has $\bar f(\alpha ) = 0$ and $\bar f' (\alpha )\neq 0$,
then there is $a\in A$ with $f(a) = 0$ and $\bar a = \alpha$.
\end{itemize}
Our main result is:
\begin{theorem}
  Let $A$ be a valuation ring with $2\in A^\times$,
  and denote by $\k$ and $K$ the residue field and the quotient
  field of $A$, respectively.
  Assume that either
  \begin{enumerate}[label=(\roman*)]
  \item $A$ is henselian and $P(\k ) < \infty$; or
  \item $P(K)<\infty$.
  \end{enumerate}
  Then the Euclidean orthogonal groups $ O_n(A)$
  satisfy homology stability with constant, abelian, and polynomial
  coefficients.
\end{theorem}
More detailed statements of this result are given in 
Theorems~\ref{thm:theorema},~\ref{thm:theorema-abelian},
and~\ref{thm:poly-stab} below.
In particular, we answer a question of Cathelineau~\cite{cathelineau2007},
which asks for the precise stability ranges that
can be achieved using Vogtmann's methods.
Theorem A also partially answers another question posed by Cathelineau
in the same article,
which asks for a good class of local rings, with infinite 
Pythagorean\footnote{
  Recall that a field $F$ is said to be \emph{Pythagorean}
  if $P(F)=1$.
}
residue fields, whose Euclidean orthogonal groups have homology
stability.

In particular, when $F$ is a field with $P(F)<\infty$, 
the fact that the groups $O_n(F)$ satisfy homology stability
with twisted coefficients is not found in the literature,
although it follows immediately from the observation that
the arguments of \cite{vogtmann82} fit into the stability
framework of \cite{RWW17}.
\subsection*{Stiefel complexes}
Concretely, we get our stability results by showing
that the \emph{Stiefel complexes} $X(\E_A^n)$
of a valuation ring $A$
are highly-connected, subject to arithmetic conditions on the
residue or quotient field of $A$.
Here $X(\E^n_A)$ is the simplicial complex whose simplices
are \emph{orthonormal frames} in $\E^n_A$ (see Section~\ref{sec:connect}).

In proving our connectivity estimates,
we use various basic results from the theory of quadratic forms
over semi-local rings.\footnote{
  Recall that a commutative ring $R$ is \emph{semi-local} if
  it only has finitely many maximal ideals.
}
As this material is non-standard, we give a brief overview
of the necessary results in Section~\ref{sec:homog-quad-R}.
The general trend of the theory, however, is that many
fundamental results for quadratic forms over fields
admit extensions to the semi-local setting.

\subsection*{Computations}
By \cite{djament-vespa}, the stable homology $\varinjlim_n H_*(G_n;F_n)$
of a finite-degree coefficient system $F$ along a sequence of groups
$G_1\to G_2\to\dots$ can often be computed in terms of two factors:
(1) the functor homology $H_*(\coprod_n BG_n;F)$ and 
(2) the homology of $G_\infty$ with constant coefficients.
In \cite{djament2012}, Djament uses this to give various general
results about the stable homology of split-orthogonal and Euclidean
orthogonal groups in a finite-degree coefficient system.
In the Euclidean case, these results can be used to compute the
stable homology of the types of coefficient systems that appear,
for $R=\R$,
in the work of Dupont and Sah on scissors congruences.

Combining Djament's stable calculations with our stability results,
it is possible to give various generalizations of computations
that, in the real case, appear in scissors congruences.
Note, however, that plugging $\R$ into our results does not give
sharp enough ranges to carry out the arguments in e.g. \cite{dupont-sah}.

We follow Vogtmann in defining the
ad hoc invariant $m_K$ for $K$ a field in order to get sharper ranges
than those which follow from a finite Pythagoras number;
by definition, $m_K$ is the smallest number $m$
(possibly $m=\infty$) such that any \emph{positive-definite}
quadratic module $(V,q)$ over $K$ contains a unit vector
if $\dim_KV\geq m$, where positive-definite means that
$q(v)$ is a sum of squares for each $v\in V$.
Pythagorean fields $K$ have $m_K = 1$,
finite fields have $m_{\F_q} = 2$, and local or global fields $K$
all have $m_K\leq 4$.
\begin{corollary}
  Let $K$ be a field of characteristic zero
  with $P(K)<\infty$,
  and let $d\in\N$.
  Then
  $$
  H_i(O_n(K);\Lambda_\Q^d(K^n))
  = 0
  $$
  for
  \begin{enumerate}[label=(\alph*)]
  \item
    $
    3i\leq n-m_K-3d-4 
    $
    if $m_K<\infty$;
  \item
    $
    2i\leq n-m_K-2d-2
    $
    if $K$ is formally real and $m_K<\infty$;
  \item 
    $
    (2P(K ) + 1)i\leq n-3-2P(K ) - (d+1)(2P(K ) + 1)
    $
    if $P(K)<\infty$; and
  \item $(P(K ) + 1)i\leq n-2-P(K ) - (d + 1)(P(K ) + 1)$
    if $K$ is formally real and $P(K)<\infty$.
  \end{enumerate}
\end{corollary}
To see this, note that $\lbrace\wedge^d_\Q (K^n)\rbrace_n$ is a degree $d$ coefficient
system, and apply Theorem~\ref{thm:poly-stab} and \cite[Thm 4]{djament2012}.

For $K$ a field, we let $\mathfrak o_n(K)$ denote the adjoint representation
of $O_n(K)$; in other words, the underlying $K$-vector space of $\mathfrak o_n(K)$
is the space of skew-symmetric $(n\times n)$-matrices,
and $A\in O_n(K)$ acts on $\mathfrak o_n(K)$ by sending
a skew-symmetric matrix $B\in\mathfrak o_n(K)$
to the skew-symmetric matrix $ABA^{-1}$.
Then we have:
\begin{corollary}
  Let $K$ be a field of characteristic zero
  with $P(K)<\infty$.
  Then
  $$
  H_i(O_n(K);\mathfrak o_n(K))
  \cong
  \bigoplus_{s+2t+1=i}H_s(O_\infty (K);\Q )\otimes \Omega_{K/\Q }^{2t+1},
  $$
  where $\Omega_{K/\Q}^*$ is the
  graded $K$-algebra of K\"ahler differentials
  over $\Q$,
  for
    \begin{enumerate}[label=(\alph*)]
  \item
    $
    i\leq \frac{n-m_K-10}{3}
    $
    if $m_K<\infty$;
  \item
    $
    i\leq \frac{n-m_K-6}{2}
    $
    if $K$ is formally real and $m_K<\infty$;
  \item 
    $
    i\leq \frac{n-3-2P(K )}{2P(K ) + 1} - 3
    $
    if $P(K)<\infty$; and
  \item $i\leq \frac{n-2-P(K )}{P(K ) + 1} - 3$
    if $K$ is formally real and $P(K)<\infty$.
  \end{enumerate}
\end{corollary}
To see this, note that $\lbrace\mathfrak o_n(K)\rbrace_n$ is a degree $2$
coefficient system, and apply Theorem~\ref{thm:poly-stab} and
\cite[Cor 6.6]{djament2012}.

As a final corollary:
\begin{corollary}
  Let $p$ be an odd prime
  and let $O_n(\Z_{(p)})$ act on $\Z_{(p)}^n$ in the canonical way.
  Then
  $$
  H_i(O_n(\Z_{(p)});\Z_{(p)}^n)= 0\quad\text{for } i\leq \frac{n-8}{2}.
  $$
\end{corollary}
To see this, note that $\lbrace\Z_{(p)}^n\rbrace_n$ is a degree $1$ coefficient
system, and apply Theorem~\ref{thm:poly-stab} and \cite[Thm 3]{djament2012}.

\medskip

\noindent\emph{Organization of the paper.}
In Section~\ref{sec:homog-quad-R}, we give an overview of the results
which we import from the theory of quadratic forms over semi-local rings.
In Section~\ref{sec:connect}, we show high connectivity for the Stiefel
complexes of a valuation ring with nice arithmetic
properties.
In Section~\ref{sec:hom-stab},
we recall the main result of~\cite{RWW17}
and show that its conditions are satisfied for the problem studied here,
thereby proving our main results.
In Appendix~\ref{sec:arithm-invars} we prove various lemmata about
arithmetic properties of local rings which are used in
Section~\ref{sec:connect}.

\medskip

\subsection*{Acknowledgements}
I was partially supported by the Danish
National Research Foundation through the Copenhagen Centre for
Geometry and Topology (DRNF151).
An earlier draft of this paper was submitted as my master's thesis
at the University of Copenhagen in December 2021.
I am deeply grateful to my mentor Nathalie Wahl,
who suggested the topic and supervised my thesis,
and to Kasper Andersen for his thorough corrections and comments.
I am also thankful to Jesper Grodal for unstucking me by suggesting that I limit
the scope of my project to hereditary local rings,
which led to the current semi-hereditary version.

\tableofcontents
\section{Quadratic forms over local rings}
\label{sec:homog-quad-R}
In other sections of this thesis,
we frequently use various facts about quadratic forms over rings.
Some of these hold over general rings,
and some are true over specifically over local or semi-local rings.
Apart from recalling basic notions and establishing notation,
the purpose of this section is to give an overview of the
results that we use above, e.g. when proving connectivity of
Stiefel complexes in Section~\ref{sec:connect}.

% One can view the stability problem for $\lbrace O_n(R)\rbrace_1^\infty$
% as a stabilization problem for automorphism groups
% inside the symmetric monoidal category of non-singular
% quadratic spaces over $R$, which we define in Sections~\ref{subsec:quad-basic}
% and~\ref{subsec:monoid} below.
% In order to apply the stability framework of
% Randal-Williams--Wahl~\cite{RWW17},\note{This should just be defined
%   somewhere in the main text.}
% we also show in Section~\ref{subsec:homogeneity} that this category
% is \emph{homogeneous} if $R$ is semi-local and $2\in R^\times$.
% \note{Maybe I should show homogeneity in the homological stability section.}
The theory of quadratic forms over semi-local rings was pioneered by
Roy, Kneser, Knebusch and others in the 60's and 70's.
The general trend of this development was that many of the
facts that Witt proved for quadratic forms over fields can be
generalized to the semi-local setting.
A good textbook reference for this theory is Baeza's book~\cite{baeza78}.
For the Witt--Pfister theory of quadratic forms over fields,
we refer to Lam's book~\cite{lam-textbook},
but also to~\cite{omeara} which contains a proof of the
general Hasse--Minkowski theorem.
\subsection{Basic facts and definitions}\label{subsec:quad-basic}
Let $R$ be a commutative ring.
\begin{defn}
  A \emph{quadratic module} over $R$ is a pair
  $(V,q)$ where $V$ is a finitely-generated projetive $R$-module
  and $q$ is a \emph{quadratic form} on $V$, meaning a function
  $q\colon V\to R$ that satisfies
  \begin{enumerate}[label=(\roman*)]
  \item $q(a x ) = a^2q(x)$ for each $a\in R$,
    $x\in V$.
  \item The function $B_q\colon V\times V\to R$ defined by
    $B_q(x,y) = q(x + y) - q(x) - q(y)$ is bilinear.
  \end{enumerate}
  We say that $(V,q)$ is \emph{non-singular} if $B_q$ is, i.e.
  if its tensor-hom adjoint
  $$
  \begin{tikzcd}[contains/.style = {draw=none,"\in" description,sloped},row sep = tiny]
    V\arrow[r,"d_q"]
    & V^*\\
    x\arrow[r,mapsto]\arrow[u,contains]
    & B_q(x,{-})\arrow[u,contains]
  \end{tikzcd}
  $$
  is an isomorphism.
  Otherwise, we say that $(V,q)$ is \emph{singular}.
\end{defn}
If $S$ is an $R$-algebra, then any quadratic module $(V,q)$ over $R$
has an associated quadratic module $(V_S,q_S)$ over $S$ with
$V_S = S\otimes_R V$
and
\begin{align*}
  &q_S(s\otimes x) = s^2q(x), \\
  &B_{q_S}(\sum_is_i\otimes x_i,\sum_js_j'\otimes x_j')
                  = \sum_i\sum_js_is_j'B_q(s_i,s_j).
\end{align*}
Given a prime ideal $\mathfrak p$ of $R$, we let
$(V(\mathfrak p),q(\mathfrak p)) = (V_{R/\mathfrak p},q_{R/\mathfrak p})$.
Elementary commutative algebra gives the following useful proposition:
\begin{prop}\label{prop:singular-reduc}
  A quadratic module $(V,q)$ over $R$
  is non-singular if and only if $(V(\mathfrak m),q(\mathfrak m))$
  is non-singular for each maximal ideal $\mathfrak m$ of $R$.
\end{prop}
\begin{exmp}
  Let $a\in R$.
  The function $R\to R$ given by $x\mapsto ax^2$ is a quadratic form
  and will be denoted by $\langle a\rangle$.
  Note that $B_{\langle a\rangle}(x,y) = a(x+y)^2 - ax^2 - ay^2 = 2axy$.
  If $2\in R^\times$, then $\langle a\rangle$ is non-singular
  if and only if $a\in R^\times$.
  Otherwise $\langle a\rangle$ is always singular.
\end{exmp}
A \emph{quadratic submodule} of a quadratic module $(V,q)$
is a direct summand $U\subseteq V$ viewed as a quadratic module
equipped with the restricted form $q\vert_U$.
\begin{exmp}
Let $(V,q)$ be a quadratic module over $R$.
Say that $x,y\in V$ are \emph{orthogonal} (under $q$)
if they are so with respect to $B_q$,
that is if $B_q(x,y) = 0$.
If $S\subseteq V$ is any set, put
$$
S^\bot = \lbrace
x\in V\mid B_q(x,s) = 0\quad\forall\, s\in S
\rbrace,
$$
i.e. the set of elements of $V$ that are orthogonal to all of $S$.
Then $S^\bot$ is referred to as the \emph{orthogonal
  complement} of $S$.
If $(V,q)$ is non-singular 
and $S=U$ is a quadratic submodule of $V$, then we have a short
exact sequence
  $$
  \begin{tikzcd}[contains/.style = {draw=none,"\in" description,sloped},row sep = tiny]
    0\arrow[r]
    &U^\bot\arrow[r]
    &V\arrow[r]
    & U^*\arrow[r]
    &0 \\
    &&x\arrow[r,mapsto]\arrow[u,contains]
    & B_q(x,{-})\vert_U \arrow[u,contains]
  \end{tikzcd},
  $$
  showing that $U^\bot$ is again a quadratic submodule of $V$.
\end{exmp}
\begin{prop}\label{prop:non-sing-summand}
  Let $(V,q)$ be a quadratic module over $R$
  and let $U\subseteq V$ be a finitely-generated projective
  submodule.
  If $(U,q\vert_U)$ is non-singular,
  then $V = U\oplus U^\bot$.
  In particular, $U$ is a quadratic submodule of $V$.
\end{prop}
\begin{proof}
  The fact that $U\cap U^\bot = 0$ follows from non-singularity of $U$.
  We show that $U + U^\bot = V$.
  Let $x$ be an arbitrary element of $V$.
  Then $B_q(x,{-})\vert_U$ is an element of $U^*$.
  Since $U$ is non-singular,
  we must therefore have $z\in U$ with $B_q(x,y) = B_q(z,y)$ for all $y\in U$.
  Hence $x-z\in U^\bot$ and $x = z + (x-z)\in U + U^\bot$.
\end{proof}
Non-singular quadratic modules over $R$ form a category $\Quad R$
in which a morphism from $(V,q)\to (V',q')$ is a form-preserving
$R$-linear map $f\colon V\to V'$, meaning that $q'(f(v)) = q(v)$
for all $v\in V$.
An isomorphism in $\Quad R$ is called an \emph{isometry},
and isomorphic quadratic modules are said to be \emph{isometric}.
We write
$$
O(q) = \Aut_{\Quad R}((V,q))
$$
for the group of self-isometries of a non-singular quadratic module $(V,q)$.
\begin{exmp}
  Let $(V,q)$ be a quadratic module over $R$
  and suppose that $v\in V$ has $q(v)\in R^\times$.
  The $R$-linear map
  $$
  \tau_v\colon x\mapsto x - \frac{B_q(x,v)}{q(v)} v
  $$
  is an element of $O(q)$.
  Note that $\tau_v^2 = \text{id}$;
  in fact, $\tau_v$ maps $v$ to $-v$
  and restricts to the identity
  on $v^\bot$.
  Geometrically,
  the map $\tau_v$ thus corresponds to reflecting across the hyperplane
  $v^\bot$.
\end{exmp}
A celebrated theorem of Cartan and Dieudonn\'e says that if $F$ is
a field with $\character F\neq 2$ and $(V,q)$ is a non-singular
quadratic space over $F$,
then $O(q)$ is generated by hyperplane reflections.
More generally, Klingenberg~\cite{klingenberg} has shown:
\begin{thm}[Cartan--Dieudonn\'e--Klingenberg]
  Let $A$ be a local ring with $2\in A^\times$.
  For any non-singular quadratic module $(V,q)$ over $A$,
  the isometry group $O(q)$ is generated by hyperplane reflections.
\end{thm}
It follows that $H_1(O_n(A);M ) = 0$ for any $O_n(A)$-module $M$
for which $M\xrightarrow{2\cdot} M$ is an isomorphism.
% If $A$ is a local ring with residue field $\k$
% and $(M,q)$ is a quadratic module over $A$,
% then we let $(\redu M,\redu q)$ denote its \emph{reduction},
% which is the quadratic space over $\k$ having
% $\redu M = \k\otimes_A M$ and $\redu q(\overline x) = \redu{q(x)}$.
% With this notation, we have the following corollary of the
% Cartan--Dieudonn\'e theorem for fields:
% \begin{cor}
%   Let $A$ be a local ring
%   with maximal ideal $\mathfrak m\not\ni 2$ and residue field $\k$,
%   and let $(V,q)$ be a non-singular quadratic space over $A$.
%   Then the group homomorphism
%   $$
%   \begin{tikzcd}[row sep = tiny]
%     O(q)\arrow[r]
%     & O(\redu q)\\
%     f\arrow[r,mapsto]
%     & \bar f = \k\otimes_A f 
%   \end{tikzcd}
%   $$
%   is surjective.
% \end{cor}
\subsection{Monoidal structure}\label{subsec:monoid}
The category $\Quad R$ has a symmetric monoidal structure called
\emph{orthogonal sum}, which we will denote by $\orthosum$.
The quadratic module $(V_1,q_1)\orthosum (V_2,q_2)$ has underlying module
$V_1\oplus V_2$ and form $q_1\orthosum q_2\colon V_1\oplus V_2\to R$
given by
$$
(q_1\orthosum q_2)((x_1,x_2)) =  q_1(x_1) + q_2(x_2)
$$
for $x_1\in V_1$ and $x_2\in V_2$.
The unit object of $\Quad R$ is the trivial $R$-module equipped
with the zero form.
Given maps of quadratic modules $f_i\colon (V_i,q_i)\to (V_i',q_i')$
for $i=1,2$,
the map $f_1\orthosum f_2$ is simply $f_1\oplus f_2\colon V_1\oplus V_2
\to V_1'\oplus V_2'$,
which one may check is form-preserving.
The required natural transformations in the symmetric monoidal structure
on $\Quad R$ are simply those associated
with the symmetric monoidal structure $\oplus$ on the category
of finitely generated projective $R$-modules;
with our definition of $q\orthosum q'$,
these maps are all form-preserving.

Recall that for $a\in R$, we denote by $\langle a\rangle$ the
quadratic form on $R$ given by $x\mapsto a x^2$.
If $a_1,\dots ,a_n\in R$, we write
$\langle a_1,\dots ,a_n\rangle = \langle a_1\rangle
\orthosum\dots\orthosum\langle a_n\rangle$,
and if $a = a_1 = \dots = a_n$,
we write $n\langle a\rangle$ for short.
More generally, if $(V,q)$ is any quadratic module over $R$, we write
$$
nq = \underbrace{q\orthosum q\orthosum\dots\orthosum q}_{n \text{ times}}.
$$
\begin{exmp}
  For each $n\geq 1$, we define
  \begin{align*}
    \E_R^n &= (R^n,n\langle 1\rangle) \tag{Euclidean $n$-space}, \\
    \H_R^{2n} &= (R^{2n},n\langle 1,-1\rangle) \tag{hyperbolic $2n$-space}.
  \end{align*}
  Both of these quadratic modules are non-singular;
  more generally,
  if $2\in R^\times$, then the form $\langle a_1,\dots ,a_n\rangle$
  is non-singular if and only if $a_1,\dots ,a_n\in R^\times$.
\end{exmp}
It is well-known that if $R = F$ is a field with $\character F\neq 2$,
then any quadratic space $(V,q)$ over $F$ admits
a \emph{diagonalization},
meaning an isometry $(V,q)\cong (F^n,\langle a_1,\dots ,a_n\rangle )$
for some $a_1,\dots ,a_n\in F$.
This fact extends to semi-local rings:
\begin{namedtheorem}{Diagonalization Theorem}
  Suppose $R$ is a semi-local ring in which $2\in R^\times$.
  If $(R^n,q)$ is a non-singular quadratic module over $R$,
  then there is an isometry
  $(R^n,q)\cong (R^n,\langle a_1,\dots ,a_n\rangle )$
  for some $a_1,\dots ,a_n\in R^\times$.
  Equivalently, $R^n$ admits a basis consisting of vectors that are
  orthogonal under $q$.
\end{namedtheorem}
\begin{proof}
  See~\cite[Prop I.3.4]{baeza78}.
\end{proof}
If $R$ is a local ring, then every projective $R$-module is free,
so the theorem says that every non-singular quadratic module admits
 a diagonalization.

To prove homogeneity, we use Roy's~\cite{roy68} generalization of Witt's cancellation
theorem to quadratic forms over semi-local rings:
\begin{namedtheorem}{Cancellation Theorem}[Roy]
  Let $R$ be a semi-local ring
  and let $(V_1,q_1)$, $(V_2,q_2)$ and $(W,q)$
  be non-singular quadratic modules over $R$.
  If
  $$
  (V_1,q_1)\orthosum (W,q)\cong (V_2,q_2)\orthosum (W,q),
  $$
  then
  $$
  (V_1,q_1)\cong (V_2,q_2).
  $$
\end{namedtheorem}
\begin{proof}
  See~\cite[Cor III.4.3]{baeza78},
  or~\cite[Thm 8.1]{roy68}
  for the original proof.
\end{proof}
\subsection{A representation theorem}
Let $R$ be a commutative ring.
\begin{defn}
  Given a quadratic module $(V,q)$ over $R$ and an element $a\in R$,
  say that $(V,q)$ \emph{represents} $a$ if there is $x\in V$ with
  $$
  q(x) = a.
  $$
  Furthermore, if $x$ can be chosen to lie outside of $\mathfrak m V$ for each
  maximal ideal $\mathfrak m\subset R$,
  then $q$ is said to \emph{primitively represent} $a$.
\end{defn}
\begin{rmk}
  In $K$-theory and homology stability literature,
  it is common to consider \emph{unimodular} elements of $R^n$,
  i.e. elements $x = (x_1,\dots ,x_n)\in R^n$ so that
  $Rx_1 + \dots + Rx_n = R$.
  On the other hand, say that $x\in R^n$ is \emph{primitive}
  if it satisfies the condition appearing in the previous definition;
  that is, $x\not\in\mathfrak mR^n$ for each maximal ideal $\mathfrak m$
  in $R$. For valuation rings, these notions coincide.
  Indeed, if $R$ is any ring, then clearly a unimodular vector
  $x\in R^n$ must also be primitive; for if $x\in\mathfrak mR^n$,
  then $Rx_1 + \dots + Rx_n\subseteq\mathfrak m$.
  Suppose now that $A$ is the valuation ring of a valuation
  $\nu\colon K\to\Gamma\cup\lbrace\infty\rbrace$.
  Recall that any finitely-generated ideal in $A$ is principal;
  indeed, if $x_1,\dots ,x_n\in A$, then picking $x_i$ with
  $\nu (x_i)\leq \nu (x_j)$ for each $j$,
  we have $Ax_1 + \dots + Ax_n = Ax_i$ since
  $x_j = x_jx_i^{-1}x_i$ (we may assume that $x_1,\dots ,x_n$
  are not all zero, and so minimality implies $x_i\neq 0$)
  and $\nu (x_jx_i^{-1}) = \nu (x_j) - \nu (x_i)\geq 0$
  so $x_jx_i^{-1}\in A$.
  Thus if $x = (x_1,\dots ,x_n)\in A^n$ is primitive, then 
  $Ax_1 + \dots + Ax_n = Ax_i$ for some $i$,
  and $x_i\not\in\mathfrak m$ since $x$ is primitive,
  so we conclude that $Ax_1 + \dots + Ax_n = A$ as desired.
  In fact, the proof shows more generally that for any ring $R$
  in which every finitely-generated ideal is principal,
  primitive and unimodular elements coincide. 
\end{rmk}
Classically, one has been interested in determining whether
a given quadratic module $(V,q)$ represents a given ring element $a\in R$.
We will be interested specifically in guaranteeing
that certain quadratic modules represent the multiplicative identity $1\in R$.
Just as for quadratic modules over fields,
it is convenient to rephrase such problems as questions about whether
certain associated quadratic modules represent zero in a non-trivial way.
\begin{defn}
  Let $(V,q)$ be a quadratic module over $R$.
  Then $(V,q)$ is
  \begin{enumerate}
  \item \emph{isotropic} if it primitively represents zero,
  \item \emph{universal} if it represents every unit in $R$.
  \end{enumerate}
\end{defn}
Hence a quadratic space $(V,q)$ over a field $F$ is isotropic
if and only if there is $0\neq x\in V$ with $q(x) = 0$.
Note also that if $(V,q)$ is a quadratic module over a
commutative ring $R$ which represents a unit $a\in R^\times$,
then $a$ is automatically primitively represented.

We will frequently use the following theorem:
\begin{namedtheorem}{Representation Theorem}
  Let $R$ be a semi-local ring
  and let $(V,q)$ be a non-singular quadratic module over $R$.
  Then $(V,q)$ represents $a\in R^\times$ if and only if
  $q\orthosum \langle -a\rangle$ is isotropic.
\end{namedtheorem}
The proof depends on a transversality theorem which we now state.
\begin{defn}
  Let $R$ be a commutative ring and let $(V,q)$ be a quadratic module
  over $R$ which admits a decomposition
  \begin{equation}
    \label{eq:decomp}
    (V,q) = (V_1,q_1)\orthosum\dots\orthosum (V_n,q_n)
  \end{equation}
  Let $x\in V$ and write
  $$
  x = x_1 + \dots + x_n
  $$
  for the unique representation with $x_i\in V_i$ for each $i$.
  Then $x$ is called \emph{transversal} to the decomposition~\eqref{eq:decomp}
  if $q(x_i)\in R^\times$ for each $i$.
\end{defn}
\begin{namedtheorem}{Transversality Theorem}
  Let $R$ be a semi-local ring and let $(V,q)$ be a non-singular quadratic
  module over $R$ which admits a decomposition
  \begin{equation*}
    (V,q) = (V_1,q_1)\orthosum\dots\orthosum (V_n,q_n).
  \end{equation*}
  Suppose $(V,q)$ is isotropic and $n$ is even.
  Then there exists $x\in V$ which is transversal to the decomposition
  and has $q(x) = 0$.
\end{namedtheorem}
\begin{proof}
  See \cite[Thm III.5.2]{baeza78}.
\end{proof}
\begin{proof}[Proof of the representation theorem]
  If $(V,q)$ represents $a\in R^\times$, pick $v\in V$ with $q(v)=a$.
  Then $(q\orthosum \langle -a\rangle )(v,1) = 0$
  and $(v,1)$ is clearly primitive.
  For the other direction,
  it follows from the transversality theorem 
  that we can pick $v\in V$, $x\in R^\times$ with
  $0 = (q\orthosum \langle -a\rangle )(v,x) = q(v) - ax^2$.
  Then $v/x$ has $q(v/x) = q(v)/x^2 = a$ as desired.
\end{proof}

\section{Stiefel complexes over local rings}
\label{sec:connect}
Let $R$ be a commutative ring.
Following~\cite{vogtmann82}, we consider:
\begin{defn}
  Let $(V,q)$ be a quadratic module over $R$.
  \begin{enumerate}[label=(\roman*)]
  \item The \emph{Stiefel complex}
    $X(q)$ is the simplicial complex whose vertices are the \emph{unit
      vectors} of $(V,q)$, meaning elements $v\in V$ with $q(v)=1$,
    and such that the $v_1,\dots ,v_k$
    span a simplex $\lbrack v_1,\dots ,v_k\rbrack$ of $X(q)$,
    also called a \emph{frame},
    if $B_q(v_i,v_j) = 0$ for $i\neq j$.
  \item For $k\in\N$, we let $X_k(q)$ denote the
    poset of frames $\lbrack v_1,\dots ,v_l\rbrack$ of length $l\leq k$
    under the relation
    $\lbrack v_1,\dots ,v_l\rbrack\leq\lbrack w_1,\dots ,w_m\rbrack$
    if $\lbrace v_1,\dots ,v_l\rbrace\subseteq\lbrace w_1,\dots ,w_m\rbrace$.
  \end{enumerate}
\end{defn}
Note that $|X_k(q)|$ is the barycentric subdivision of $|\sk_{k-1}X(q)|$, so
\begin{equation*}
  |X_k(q)|\simeq |\sk_{k-1}X(q)|,
\end{equation*}
for $1\leq k\leq\rank V$.

The goal of this section is to show that if $R$ is valuation ring
subject to certain arithmetic conditions,
then the complex $|X(n\langle 1\rangle )|$ is highly connected when $n$ is large.
Equivalently, we will show that $|X_k(n\langle 1\rangle )|$ is $(k-1)$-spherical for
large $k$ when $n$ is large.

Intuitively, showing connectivity of $X(n\langle 1\rangle )$
will depend on having a sufficient supply of unit
vectors in the link of any simplex.
As in~\cite{vogtmann82}, we therefore consider arithmetic invariants
that force the existence of such unit vectors.
The most obvious of these is:
\begin{defn}
  Let $A$ be a local ring with residue field $\k$.
  We define $m_A\in\N\cup\lbrace\infty\rbrace$ to be the smallest number
  $m$ such that if $V$ is a non-singular quadratic submodule of
  $\E^n_A$ for some $n$ and $\dim_{\k}\redu V\geq m$,
  then $V$ contains a unit vector.
\end{defn}
Note that by the diagonalization theorem,
if $2\in A^\times$ then a non-singular quadratic module $(V,q)$ over $A$
is isometric to a quadratic submodule of $\E^n_A$
for some $n$ if and only if $V$ is non-singular and admits a
diagonalization $\langle d_1,\dots ,d_m\rangle $ in which each $d_i$
is a sum of squares.
\begin{defn}
  Let $R$ be a ring.
  We define the \emph{Pythagoras number}
  $P(R)\in\N\cup\lbrace\infty\rbrace$ to be the smallest number $p$
  such that if $a\in R$ is a sum of squares, then $a$ is a sum of $p$ squares.
\end{defn}
For a field $F$, finiteness of the Pythagoras number $P(F)$
also guarantees the existence of unit vectors in non-singular subspaces
of Euclidean spaces, albeit not as efficiently as finiteness of $m_F$.
The following is~\cite[Prop 1.5]{vogtmann82},
the proof of which we include for completeness:
\begin{lem}\label{lem:pythagoras-unit-vector}
  Let $F$ be a field of characteristic $\neq 2$
  having $P = P(F) < \infty$,
  and let $V\subseteq\E^n_F$ be a non-singular subspace of codimension $k$.
  If $n>Pk$, then $V$ contains a unit vector.
\end{lem}
\begin{proof}
  Since $V$ is non-singular,
  we have by Proposition~\ref{prop:non-sing-summand} that
  $\E^n_F\cong V\orthosum V^\bot$.
  Let $\langle a_1,\dots ,a_{n-k}\rangle\cong n\langle 1\rangle\vert_V$
  be a diagonalization of $V$.
  By picking a diagonalization of $V^\bot$ also, we get a diagonalization
  \begin{equation}
    \label{eq:vogtmann-diag1}
    n\langle 1\rangle\cong \langle a_1,\dots ,a_{n-k},b_1,\dots ,b_k\rangle 
  \end{equation}
  of $\E^n_F$.
  Each $b_i\neq 0$ is represented by $n\langle 1\rangle$, i.e. is a sum
  of squares; by assumption each $b_i$ is therefore represented by
  $P\langle 1\rangle$. Thus for each $1\leq i\leq k$, we have an isometry
  \begin{equation}
    \label{eq:vogtmann-diag2}
    P\langle 1\rangle\cong \langle b_i, x_{1,i},\dots ,x_{P-1,i}\rangle
  \end{equation}
  for some $x_{1,i},\dots ,x_{P-1,i}\in F^\times$.
  Since $n>Pk$, we get from~\eqref{eq:vogtmann-diag1} and~\eqref{eq:vogtmann-diag2}
  that
  \begin{align*}
    &\langle b_1,\dots ,b_k, a_1,\dots ,a_{n-k}\rangle \\
    &\qquad\cong n\langle 1\rangle \\
    &\qquad\cong\langle b_1,\dots ,b_k,
      x_{1,1},\dots ,x_{P-1,1},x_{1,2},\dots ,x_{P-1,2},
      \dots ,
      x_{1,k},\dots ,x_{P-1,k},\underbrace{1,\dots ,1}_{n-Pk\text{ times}}\rangle,
  \end{align*}
  so by Witt cancellation
  $$
  \langle  a_1,\dots ,a_{n-k}\rangle
  \cong
  \langle x_{1,1},\dots ,x_{P-1,1},x_{1,2},\dots ,x_{P-1,2},
  \dots ,
  x_{1,k},\dots ,x_{P-1,k},\underbrace{1,\dots ,1}_{n-Pk\text{ times}}\rangle,
  $$
  and in particular $n\langle 1\rangle\vert_V\cong \langle a_1,\dots ,a_{n-k}\rangle$
  represents one as desired.
\end{proof}
% In order to use something like the previous lemma or the definition of $m_A$
% to prove connectivity of Stiefel complexes,
% we will need to know that certain intersections of non-singular subspaces
% contain sufficiently high-dimensional non-singular subspaces.
% For this, we use Kneser's generalization of the Witt--Arf decomposition
% theorem to quadratic modules over semilocal rings.
% Note in particular that the $(V_h,q_h)\orthosum (V_a,q_a)$ summand
% is non-singular
\begin{prop}
For a field $F$, 
\begin{equation}
  \label{eq:p-leq-m}
  P(F)\leq m_F.
\end{equation}
\end{prop}
\begin{proof}
  If $\operatorname{char} F = 2$, then
  $x_1^2+\dots +x_n^2 = (x_1+\dots +x_n)^2$
  for all $x_1,\dots ,x_n\in F$,
  so $P(F) = 1$, and hence the inequality holds trivially.

  Otherwise, assuming $m_F<\infty$ and that $0\neq a\in F$ is a sum of $n>m_F$
  squares, pick $w\in\E^n_F$ with $q(w) = a$,
  where we write $q = n\langle 1\rangle$ for short.
  Then $w^\bot\subset\E^n_F$ has dimension $n-1\geq m_F$,
  so there is $v\in w^\bot$ with $q(v)=1$.
  Diagonalizing $\E^n_F$ starting with $\lbrace v,w\rbrace$,
  we find that $\E^n_F\cong\langle 1,a,d_1,\dots ,d_{n-2}\rangle$.
  By Witt cancellation, 
  we then have $\E^{n-1}_F\cong\langle a,d_1,\dots ,d_{n-2}\rangle$,
  so $a$ is a sum of $n-1$ squares. 
  Proceeding by induction, we find that $a$ is a sum of $m_F$ squares.
\end{proof}
In \cite{vogtmann82}, Vogtmann credits the following result, 
which generalizes the previous proposition,
to Daniel Shapiro:\footnote{
  The proof given by Vogtmann does not work in general, and it
  was pointed out to me by Kasper Andersen that in fact the statement
  of the proposition is false for $F = \mathbb F_3$.
  Nevertheless, $F = \mathbb F_3$ is the only field for which the
  statement fails.
}
\begin{prop}
  Let $F\neq\mathbb F_3$ be a field of characteristic $\neq 2$.
  Suppose there is $k$ so that for each pair of unit vectors
  $e,f\in\E^n_F$, $n\geq k$, the subspace $e^\bot\cap f^\bot$ contains
  a unit vector. Then $P(F)\leq k-2$.
\end{prop}
\begin{proof}
  Let $a\in F$. We claim that we can write $a = x^2-y^2$
  where $x\neq 0$. For $a=0$ this is obvious, as we can
  take $x = y = 1$.
  For $a\neq 0$, we claim that $-a \neq z^2$ for some $z\in F^\times$.
  As there are at most two $z$ with $z^2 = - a$,
  this is possible as long as $\# F^\times > 2$,
  and the latter holds since $F$ is neither $\mathbb F_2$ or $\mathbb F_3$.
  Then
  $$
  a = \left(\frac{z+az^{-1}} 2\right)^2 - \left(\frac{z-az^{-1}} 2\right)^2
  $$
  as claimed.

  Assume $a\in F$ is a sum of $n \geq k-1$ squares for some $n$.
  As shown above, we may write $a = x^2 - y^2$ with $x\neq 0$.
  Write $c = a/x^2$ and note that $c = 1 - (y/x)^2$.
  Since $a$ is a sum of $n$ squares, so is $c$.
  Pick $w\in\E^n_F$ with $q(w) = c$, where $q = n\langle 1\rangle$.
  We view $\E^n_F$ as a subspace of $\E^{n+1}_F$,
  and let $e$ denote the $(n+1)$st standard basis element of $\E^{n+1}_F$.
  In particular, $q(e)=1$.
  Putting $f = (y/x)e+w$, we then have $q(f) = (y/x)^2 + 1 - (y/x)^2 = 1$.

  We have $n+1\geq k$,
  so by assumption $e^\bot\cap f^\bot = e^\bot\cap w^\bot\cap f^\bot$ 
  contains a unit vector $v$.
  Diagonalize $\E^n_F$ starting with the vectors $e,v,w$,
  so that $\E^{n+1}_F\cong\langle 1,1,c,d_1,\dots ,d_{n-3}\rangle$.
  By Witt cancellation we thus have
  $\E^{n-1}_F\cong\langle c,d_1,\dots ,d_{n-3}\rangle$.
  It follows that $c$ is a sum of $n-1$ squares.
  Proceeding inductively, we find that $c$ is a sum of $k-1$ squares.
\end{proof}
\begin{lem}\label{lem:split-off-kernel}
  Let $R$ be a semi-hereditary\footnote{
    Recall that a commutative ring $R$ is \emph{semi-hereditary}
    if every finitely-generated submodule of a projective module is projective.
  }ring, and let
  $$
  0\to K\to E\xrightarrow f F
  $$
  be an exact sequence of $R$-modules,
  where $E$ and $F$ are projective
  and $E$ is finitely-generated.
  Then $K$ is a direct summand of $E$.
\end{lem}
\begin{proof}
  The image $F' = f(E)\subseteq F$ is finitely-generated,
  so $F'$ is projective by~\cite{albrecht61}.
  Hence
  $$
  0\to K\to E\xrightarrow f F'\to 0
  $$
  splits.
\end{proof}
Recall that semi-hereditary local rings are the same 
as valuation rings. (Indeed~\cite[Thm 64]{kaplansky74}
says that an integral domain $R$ is semi-hereditary if and only
if $R_{\mathfrak m}$ is a valuation ring for each maximal ideal $\mathfrak m$
of $R$; and~\cite{jensen66} shows that a semi-hereditary local ring
is integral.)
\begin{lem}\label{lem:splitting-off-radical}
  Let $A$ be a local ring with maximal ideal $\mathfrak m$,
  and let $(V,q)$ be a quadratic $A$-module.
  Then $V$ admits a decomposition
  $$
  V\cong R\orthosum W,
  $$
  where $W$ is free and non-singular,
  and $q(x)\in\mathfrak m$ for each $x\in R$.
\end{lem}
\begin{proof}
  Over $\k = A/\mathfrak m$,
  we have $\redu V\cong\mathbf R\orthosum\mathbf W$,
  where $\mathbf R$ is the radical of $\redu V$
  and $\mathbf W$ is non-singular.
  By \cite[Cor I.3.4]{baeza78},
  we then have a decomposition $V\cong R\orthosum W$
  where $\redu W = \mathbf W$
  and $\redu R = \mathbf R$.
  It follows that $W$
  is non-singular and $\overline{q(x)} = 0$, i.e. $q(x)\in\mathfrak m$,
  for each $x\in R$.
\end{proof}
Given a quadratic module $(V,q)$ over $A$,
let us write $W\leq V$ when $W$ is a quadratic submodule of $V$,
i.e. $W$ is a direct summand in $V$.
Note that this defines a partial order on the set of submodules of $V$.
\begin{lem}\label{lem:large-nonsing-subspace}
  Let $A$ be a valuation ring with $2\in A^\times$,
  and let $\lbrack u_1,\dots ,u_r\rbrack$ and $\lbrack v_1,\dots ,v_s\rbrack$
  be possibly empty frames in $\E^n_A$ spanning
  quadratic submodules $U$ and $V$ respectively.
  Then $U^\bot\cap V^\bot$ contains a non-singular subspace
  of dimension at least $n-r-2s$.
\end{lem}
\begin{proof}
  By definition, there is an exact sequence
  \begin{equation}
    \label{eq:intersection-def-exact-seq}
    0\to U^\bot\cap V^\bot\to U^\bot\to V^*.
  \end{equation}
  It follows from Lemma~\ref{lem:split-off-kernel} that
  $U^\bot\cap V^\bot\leq U^\bot$.
  Pick a decomposition
  $$
  U^\bot\cap V^\bot\cong R\orthosum W
  $$
  as in Lemma~\ref{lem:splitting-off-radical}.
  Since $R\leq U^\bot$ and $U^\bot$ is non-singular,
  there is an $S\leq U^\bot$ which maps isomorphically to
  $R^*\leq (U^\bot)^*$ under the duality isomorphism
  $$
  \begin{tikzcd}[row sep = tiny]
    U^\bot\arrow[r,"\delta"]
    & (U^\bot )^*\\
    x\arrow[r,mapsto]
    & B_{\langle n\rangle }(x,{-})\vert_{U^\bot}.
  \end{tikzcd}
  $$
  We claim that $S\cap V^\bot = 0$.
  Indeed, note that by exactness of
  $$
  0\to S\cap V^\bot\to S\to V^*
  $$
  and Lemma~\ref{lem:split-off-kernel}, we have that
  $S\cap V^\bot\leq S$.
  Hence $\delta (S\cap V^\bot)\leq R^*$.
  But $S\cap V^\bot\subseteq U^\bot\cap V^\bot$,
  so by construction of $R$ we have $\delta (S\cap V^\bot)\subseteq\mathfrak mR^*$.
  Since $\delta (S\cap V^\bot)$ is a direct summand of $R^*$
  but is contained in $\mathfrak mR^*$, we must have $\delta (S\cap V^\bot) = 0$,
  and hence $S\cap V^\bot = 0$.
  But then
  $$
  \dim R^* + n - s = \dim S + \dim V^\bot = \dim (S+V^\bot)\leq n,
  $$
  so $\dim R = \dim R^*\leq s$.
  It follows that $\dim W = \dim (U^\bot\cap V^\bot) - \dim R\geq n-r-2s$,
  proving the lemma.
\end{proof}
\begin{prop}\label{prop:unit-in-intersection}
  Let $A$ be a valuation ring with residue field $\k$,
  and assume that $2\in A^\times$.
  Let $\lbrack u_1,\dots ,u_r\rbrack$ and $\lbrack v_1,\dots ,v_s\rbrack$
  be possibly empty frames in $\E^n_A$ spanning
  quadratic submodules $U$ and $V$ respectively.
  Then $U^\bot\cap V^\bot$ contains a unit vector if either
  \begin{enumerate}[label=(\roman*)]
  \item $n\geq m_A+r+2s$; or
  \item $n\geq m_A + r+s$ and $\k$ is formally real; or
  \item $n > 2P(\k )r + s$ and $A$ is henselian; or
  \item $n > P(\k )r + s$, $A$ is henselian, and $\k$ is formally real.
  \end{enumerate}
\end{prop}
\begin{proof}
  As in the proof of the previous lemma,
  we find that $U^\bot\cap V^\bot\leq\E^n_A$.
  If $\k$ is formally real, then $U^\bot\cap V^\bot$
  is non-singular and (ii) follows from the definition of $m_A$.
  Similarly, (iv) follows from Lemma~\ref{lem:pythagoras-unit-vector}
  and Lemma~\ref{lem:hensel-isotropy}.
  
  In the general case, Lemma~\ref{lem:large-nonsing-subspace}
  guarantees the existence of a non-singular quadratic subspace
  $W\leq U^\bot\cap V^\bot\leq\E^n_A$ with $\dim W\geq n-r-2s$.
  Hence (i) follows from the definition of $m_A$
  and (iii) follows from Lemma~\ref{lem:pythagoras-unit-vector}
  and Lemma~\ref{lem:hensel-isotropy},
  along with the fact that $W$ is a codimension $2s$ subspace of $U^\bot$
  and that $U^\bot\cong\E^{n-r}_A$,
  as one sees by applying Witt cancellation to the decomposition
  $\E^r_A\orthosum U^\bot\cong U\orthosum U^\bot\cong\E^n_A$.
\end{proof}
\begin{prop}\label{prop:unit-in-intersection2}
  Let $A$ be a valuation ring with $2\in A^\times$,
  and denote the quotient field of $A$ by $K$.
  Let $\lbrack u_1,\dots ,u_r\rbrack$ and $\lbrack v_1,\dots ,v_s\rbrack$
  be possibly empty frames in $\E^n_A$ spanning
  quadratic submodules $U$ and $V$ respectively, and with $r\geq s\geq 0$.
  Then $U^\bot\cap V^\bot$ contains a unit vector if either
    \begin{enumerate}[label=(\roman*)]
  \item $n\geq m_K+2r+s$; or
  \item $n\geq m_K+r+s$ and $K$ is formally real; or
  \item $n > 2P(K )r + s$; or
  \item $n > P(K)r + s$ and $K$ is formally real.
  \end{enumerate}
\end{prop}
\begin{proof}
  The proof is analogous to the proof of the previous proposition,
  using Lemma~\ref{lem:DVR-isotropy} to push unit vectors over $K$
  into unit vectors over $A$.
\end{proof}
\begin{thm}\label{thm:intersection-connect}
  Let $A$ be a valuation ring with $2\in A^\times$,
  and denote the residue and quotient fields of $A$
  by $\k$ and $K$, respectively.
  Let $\lbrack u_1,\dots ,u_r\rbrack$ and $\lbrack v_1,\dots ,v_s\rbrack$
  be possibly empty frames in $\E^n_A$ spanning
  quadratic submodules $U$ and $V$ respectively, and with $r\geq s\geq 0$.
  Then
  $$
  | X_l(U^\bot\cap V^\bot ) |\simeq\bigvee S^{l-1}
  $$
  if either
  \begin{enumerate}[label=(\roman*)]
  \item $A$ is henselian and $n\geq 2(r+l-1) + (s+l-1) + m_A$;
  \item $A$ is henselian and $n\geq 2P(\k )(r+l-1) + (s+l-1) + 1$;
  \item $A$ is henselian, $\k$ is formally real,
    and $n\geq (r+l-1) + (s+l-1) + m_A$;
  \item $A$ is henselian, $\k$ is formally real,
    and $n\geq P(\k )(r+l-1) + (s+l-1) + 1$;
  \item $n\geq 2(r+l-1) + (s+l-1) + m_K$;
  \item $n\geq 2P(K ) (r+l-1) + (s + l -1)+1$;
  \item $K$ is formally real and $n\geq (r+l-1) + (s+l-1) + m_K$; or
  \item $K$ is formally real and
    $n\geq P(K ) (r+l-1) + (s + l -1)+1$.
  \end{enumerate}
\end{thm}
\begin{cor}\label{cor:cnt-ranges}
  Let $A$ be a valuation ring with $2\in A^\times$,
  and denote the residue field of $A$ by $\k$.
  The Stiefel complex $X(n\langle 1\rangle )$ is
  \begin{enumerate}[label=(\roman*)]
  \item $\left(\frac{n-m_A-3} 3\right)$-connected if $m_A<\infty$;
    \item $\left(\frac{n-5-2P(\k)}{2P(\k)+1} \right)$-connected
      if $P(\k ) < \infty$;
  \item $\left(\frac{n-m_A-2} 2\right)$-connected
    if $m_A <\infty$; and
  \item $\left(\frac{n-4-P(\k )}{P(\k ) +1}\right)$-connected
    if $A$ is henselian and $P(\k )<\infty$.
  \end{enumerate}
  Further, if $K$ denotes the quotient field of $A$,
  then the Stiefel complex $X(n\langle 1\rangle )$ is
  \begin{enumerate}[label=(\roman*)]
    \setcounter{enumi}{4}
    \item $\left(\frac{n-m_K-3} 3\right)$-connected if $m_K<\infty$;
    \item $\left(\frac{n-5-2P(K)}{2P(K)+1} \right)$-connected
      if $P(K) < \infty$;
    \item $\left(\frac{n-m_A-2} 2\right)$-connected
      if $K$ is formally real and $m_K<\infty$; and
    \item $\left(\frac{n-4-P(\k )}{P(\k ) +1}\right)$-connected
      if $K$ is formally real and $P(K)<\infty$
  \end{enumerate}
\end{cor}
The proof uses three lemmata from the homotopy theory of posets:
\begin{lem}[Discrete Morse theory]\label{lem:disc-morse-thy}
  Let $X$ be a poset with $X = X_0\cup L_1\cup\dots\cup L_n$
  as sets. Put $X_j = X_0\cup L_1\cup\dots\cup L_j$
  for each $1\leq j\leq n$. Assume there is $d\in\N$ such that
  \begin{enumerate}[label=(\roman*)]
  \item $|X_0|\simeq\bigvee S^d$.
  \item Distinct elements of $L_i$ are not comparable for each $i\geq 1$.
  \item For each $i\geq 1$ and $x\in L_i$,
    $$
    |\Lk x X\cap X_{i-1}|\simeq\bigvee S^{d-1}.
    $$
  \end{enumerate}
  Then $|X|\simeq\bigvee S^d$.
\end{lem}
\begin{proof}
  By induction, it suffices to consider the case $n=1$.
  Using Zorn's lemma, pick a maximal subset $L'\subseteq L_1$
  with the property that $|X_0\cup L'|\simeq\bigvee S^d$.
  We claim that $L' = L_1$.
  Otherwise pick $x\in L_1\setminus L'$.
  Then
  \begin{align*}
    |X_0\cup L'\cup\lbrace x\rbrace|
    &= |X_0\cup L'|\cup_{|\Lk x X\cap (X_0\cup L')|}|\St x X\cap (X_0\cup L')| \\
    &=|X_0\cup L'|\cup_{|\Lk x X\cap X_0|}|\St x {X_0\cup\lbrace x\rbrace}|
  \end{align*}
  using that $x$ is not comparable to any element of $L'$.
  Since $|X_0\cup L'|\simeq\bigvee S^d$,
  the attaching map
  $\bigvee S^{d-1}\simeq|\Lk x X\cap X_0|\hookrightarrow |X_0\cup L'|$
  is nulhomotopic.
  Hence
  \begin{align*}
    |X_0\cup L'\cup\lbrace x\rbrace|
    &\simeq\left( \bigvee S^d\right) \vee
      \left( \cone |\Lk x X\cap X_0|
      \cup_{|\Lk x X\cap X_0|} |\St x {X_0\cup\lbrace x\rbrace} |\right) \\
    &\simeq \left( \bigvee S^d\right) \vee \susp |\Lk x X\cap X_0| \\
    &\simeq \bigvee S^d,
  \end{align*}
  contradicting the maximality of $L'$.
\end{proof}
\begin{rmk}
  The statement also holds (with exactly the same proof)
  when we allow $\bigvee S^k$ to mean the
  empty wedge of $k$-spheres, i.e. the point $*$.
\end{rmk}
\begin{lem}[Poset deformation lemma]\label{lem:poset-deformation}
  Let $X$ be a poset.
  If $f\colon X\to X$ is a poset endomorphism with $f(x)\leq x$
  for each $x\in X$, then $|f|\colon |X|\to |\im f|$
  is a homotopy equivalence.
\end{lem}
\begin{lem}\label{lem:poset-join}
  Let $X$ be a poset.
  If $X = Y\sqcup Z$ and $y\leq z$ for each $y\in Y$, $z\in Z$,
  then $|X|\simeq |Y|*|Z|$.
\end{lem}
\begin{proof}[Proof of Theorem~\ref{thm:intersection-connect}]
  In order to deal with the cases (i)-(viii) simultaneously,
  we let
  $$
  g\colon \N\times \Z_{\geq 0}\times \Z_{\geq 0}\to\N
  $$
  denote the function given
  by $g(l,r,s) = 2(r+l-1) + (s+l-1) + m_A$ in case (i),
  by $g(l,r,s) = 2P(\k )(r+l-1) + (s+l-1) + 1$ in case (ii),
  and so forth, and assume that $A$ satisfies the requirements of the
  corresponding case.

  The proof is by induction on $l$.
  For $l=1$, Proposition~\ref{prop:unit-in-intersection}
  or~\ref{prop:unit-in-intersection2} implies that if
  $n\geq g(1,r,s)$, then
  $X_1(U^\bot\cap V^\bot)$ contains a unit vector $u$.
  But then $u\neq -u\in X_1(U^\bot\cap V^\bot)$ also, so
  the discrete poset $X_1(U^\bot\cap V^\bot)$ has at least two elements, 
  and hence $|X_1(U^\bot\cap V^\bot) |\simeq\bigvee S^0$ as desired.

  Assume now that $l\geq 2$ and $n\geq g(l,r,s)$.
  Write $X = X_l(U^\bot\cap V^\bot )$ for convenience.
  Since $g$ is monotonic in each variable, we have in particular
  $$
  g(l,r,s)\geq g(1,r,s),
  $$
  so as above we find by Proposition~\ref{prop:unit-in-intersection}
  or~\ref{prop:unit-in-intersection2}
  that we can pick a unit vector $u\in U^\bot\cap V^\bot$.
  Set $W = U^\bot\cap V^\bot\cap u^\bot$. Since
  $$
  g(l,r,s)\geq g(l-1,r,s+1),
  $$
  the induction hypothesis gives
  $|X_{l-1}(W)|\simeq\bigvee S^{l-2}$.

  Write the poset $X$ as a union
  $$
  X = X_0\cup L_1\cup\dots\cup L_l
  $$
  of sets, where
  \begin{align*}
    X_0 &= \lbrace \lbrack\pm u\rbrack \rbrace
          \cup
          \lbrace\lbrack v_1,\dots ,v_i\rbrack \mid \exists\, t \geq 1,\,
          \lbrack v_1,\dots ,v_t\rbrack\in X_{l-1}(W)\,\wedge \,
          B_q(v_j,u)\neq 0\,\,\text{for}\,\, t < j\leq i\rbrace,\\
    L_1 &= \lbrace \lbrack v_1,\dots ,v_l\rbrack\in X\mid v_1,\dots ,v_l\in U^\bot\cap V^\bot\cap u^\bot\rbrace,\quad\text{and} \\
    L_i &= \lbrace\lbrack v_1,\dots ,v_i\rbrack\mid B_q(v_j,u)\neq 0
          \,\,\text{for all}\,\, j\rbrace,\quad 2\leq i\leq l.
  \end{align*}
  We verify the conditions of the discrete Morse theory lemma
  with $d=l-1$.
  For each $i\geq 1$, all frames in $L_i$ have the same length
  and thus no distinct elements of $L_i$ are comparable.
  It remains to check conditions (i) and (iii) of the lemma.
  \begin{claim}
    $|X_0|\simeq\bigvee S^d$.
  \end{claim}
  \begin{claimproof}
    Put
    \begin{align*}
      X_0' = \lbrace \lbrack \pm u\rbrack\rbrace
      \cup \lbrace \lbrack \pm u,v_1,\dots ,v_i\rbrack\in X\rbrace
      \cup \lbrace \lbrack v_1,\dots ,v_i\rbrack\in X\mid B_q(v_j,u)=0\,\,\text{for all}\,\, j\rbrace.
    \end{align*}
    Then $|X_0'| = \susp{|X_{l-1}(W)|}\simeq\bigvee S^{l-1}$.
    We define a poset map $f\colon X_0\to X_0$ which deforms $X_0$
    onto $X_0'$ as in the poset deformation lemma.
    Set $f$ to be the identity on $X_0'$.
    Suppose $\lbrack v_1,\dots ,v_i\rbrack\in X$
    and $t$
    are such that $\lbrack v_1,\dots ,v_t\rbrack\in X_{l-1}(W)$
    and $B_q(v_j,u)\neq 0$ for $t <j\leq i$.
    Then we define $f(\lbrack v_1,\dots ,v_i\rbrack ) = \lbrack v_1,\dots ,v_t\rbrack$
    and
    $f(\lbrack\pm u, v_1,\dots ,v_i\rbrack ) = \lbrack \pm u, v_1,\dots ,v_t\rbrack$,
    and $f$ is as desired.
  \end{claimproof}
  As in the Morse theory lemma,
  we put $X_j = X_0\cup L_1\cup\dots\cup L_j$
  for each $1\leq j\leq l$.
  We must check:
  \begin{claim}
    $|\Lk{\lbrack v_1,\dots ,v_i\rbrack} X\cap X_{i-1}|\simeq\bigvee S^{l-2}$
    for all $i\geq 1$
    and $\lbrack v_1,\dots ,v_i\rbrack\in L_i$.
  \end{claim}
  \begin{claimproof}
    Suppose first that $\lbrack v_1,\dots ,v_l\rbrack\in L_1$.
    Then
    $$
    \Lk{\lbrack v_1,\dots ,v_l\rbrack}{X}\cap X_0
    = \Lk{\lbrack v_1,\dots ,v_l\rbrack}{X}
    =\lbrace
    \text{proper subframes of}\,\, \lbrack v_1,\dots ,v_l\rbrack
    \rbrace.
    $$
    This can be identified with the barycentric subdivision of the
    boundary of an $(l-1)$-simplex, so
    $|\Lk{\lbrack v_1,\dots ,v_l\rbrack}{X}\cap X_0|\simeq\bigvee S^{l-2}$
    as desired.

    Now suppose that $\lbrack v_1,\dots ,v_i\rbrack\in L_i$ for some
    $i\geq 2$.
    Then
    \begin{align*}
      &\Lk{\lbrack v_1,\dots ,v_i\rbrack} X\cap X_{i-1} \\
      &\quad =
        \lbrace\text{proper subframes of}\,\, \lbrack v_1,\dots ,v_i\rbrack\rbrace \\
      &\qquad
        \cup
        \lbrace\lbrack w_1,\dots ,w_t ,v_1,\dots ,v_i\rbrack
        \mid 1\leq t\leq l-i,\, \lbrack w_1,\dots ,w_t\rbrack\in X_{l-1}(W)\rbrace \\
      &\quad =
        \lbrace\text{proper subframes of}\,\, \lbrack v_1,\dots ,v_i\rbrack\rbrace \\
      &\qquad
        \cup
        \lbrace\lbrack w_1,\dots ,w_t ,v_1,\dots ,v_i\rbrack
        \mid \lbrack w_1,\dots ,w_t\rbrack
        \in X_{l-i}(W\cap\lbrack v_1,\dots ,v_i\rbrack^\bot)\rbrace.
    \end{align*}
    Thus Lemma~\ref{lem:poset-join} implies that
    \begin{align*}
      &|\Lk{\lbrack v_1,\dots ,v_i\rbrack} X\cap X_{i-1}| \\
      &\quad\simeq
        |\lbrace\text{proper subframes of}\,\, \lbrack v_1,\dots ,v_i\rbrack\rbrace |\\
      &\qquad
        *
        |\lbrace\lbrack w_1,\dots ,w_t ,v_1,\dots ,v_i\rbrack
        \mid \lbrack w_1,\dots ,w_t\rbrack
        \in X_{l-i}(W\cap\lbrack v_1,\dots ,v_i\rbrack^\bot )\rbrace | \\
      &\quad\simeq S^{i-2}* |X_{l-i}(W\cap\lbrack v_1,\dots ,v_i\rbrack^\bot )| \\
      &\quad\simeq S^{i-2}*S^{l-i-1}\simeq S^{l-2},
    \end{align*}
    where we have used that
    $$
    g(l,r,s)\geq g(l-i,r,s+i),
    $$
    so the induction hypothesis applies to
    $X_{l-i}(W\cap\lbrack v_1,\dots ,v_i\rbrack^\bot )$.
  \end{claimproof}
  It follows from the discrete Morse theory lemma that $|X|$ is homotopy
  equivalent to a wedge of $(l-1)$-spheres, thus completing the induction.
\end{proof}

\section{Homology stability of $O_n(A)$}
\label{sec:hom-stab}
For a pair of objects $(A,X)$ in a \emph{homogeneous}
monoidal category $(\mathcal C,\oplus ,0)$,
the main result of~\cite{RWW17} says that 
the sequence of groups
$$
\text{Aut}_{\mathcal C} (A)
\xrightarrow{{-}\oplus X}
\text{Aut}_{\mathcal C} (A\oplus X)
\xrightarrow{{-}\oplus X}
\dots
\xrightarrow{{-}\oplus X}
\text{Aut}_{\mathcal C} (A\oplus X^{\oplus n})
\xrightarrow{{-}\oplus X}
\dots
$$
satisfies homology stability if
a certain associated semi-simplicial set
$W_n(A,X)_\bullet$ is sufficiently connected for large $n$.
In this section, we derive homology stability of the Euclidean orthogonal groups
$O_n(A)$ for some local rings by applying this framework to the pair
$(0,\E^1_A)$ in the category $\Quad A$ of non-singular
quadratic modules over $A$.

\subsection{Homogeneity of $\Quad R$}
In order to apply this framework, we first show that the category $\Quad A$
is homogeneous. Recall that for a monoidal category
$(\mathcal C ,\oplus ,0)$ in which the monoidal unit $0$ is initial,
this means that for all objects $A,B\in\mathcal C$:
\begin{itemize}
\item[\textbf{H1}] The action of $\text{Aut}_{\mathcal C}(B)$
  on the set $\text{Hom}_{\mathcal C}(A,B)$ by postcomposition
  is transitive.
\item[\textbf{H2}] The group homomorphism
  $$
  \begin{tikzcd}[row sep = tiny]
    \text{Aut}_{\mathcal C}(A)\arrow[r]
    & \text{Aut}_{\mathcal C}(A\oplus B)\\
    f\arrow[r,mapsto]
    & f\oplus B
  \end{tikzcd}
  $$
  is injective, 
  with image
  $\text{Fix}(0\oplus B) = \lbrace\varphi\in \text{Aut}_{\mathcal C}(A\oplus B)\mid \varphi\circ \iota_B = \iota_B\rbrace$,
  where $\iota_B$ is the canonical morphism $B\cong 0\oplus B\xrightarrow{0_A\oplus B} A\oplus B$.
\end{itemize}
Before proving homogeneity, we make the basic observation:
\begin{prop}\label{prop:preserve-bilinear}
  Let $R$ be a commutative ring
  and let $f\colon (V,q_V)\to (U,q_U)$ be a map of quadratic modules
  over $R$.
  Then for each $x,y\in V$, $B_{q_U}(f(x),f(y)) = B_{q_V}(x,y)$.
\end{prop}
\begin{proof}
  By definition
  \begin{align*}
    B_{q_U}(f(x),f(y)) &= q_U(f(x) + f(y)) - q_U(f(x)) - q_U(f(y)) \\
                       &= q_V(x + y) - q_V(x) - q_V(y) \\
                       &= B_{q_V}(x,y).
  \end{align*}
\end{proof}
\begin{cor}\label{cor:injective}
  Let $R$ be a commutative ring
  and let $f\colon (V,q_V)\to (U,q_U)$ be a map of quadratic modules
  over $R$.
  If the quadratic module $(V,q_V)$ is non-singular, then $f$ is injective.
\end{cor}
\begin{proof}
  Let $0\neq x\in V$.
  By assumption, the duality map $d_{q_V}\colon V\to V^*$ is an
  isomorphism; in particular, it is injective.
  Hence $B_{q_V}(x,{-})\neq 0$, i.e. there is $y\in V$ with
  $B_{q_V}(x,y)\neq 0$.
  By the lemma, we then have $B_{q_U}(f(x),f(y)) = B_{q_V}(x,y)\neq 0$,
  so $f(x)\neq 0$ as desired.
\end{proof}
\begin{thm}
  Let $R$ be a semi-local ring with $2\in R^\times$.
  Then the symmetric monoidal category $(\Quad R ,\orthosum ,0)$
  of non-singular quadratic modules is homogenenous.
\end{thm}
\begin{proof}
  First observe that the trivial quadratic module $0$ is initial
  in $\Quad R$.
  Condition H2 is clear from the definition of $\orthosum$;
  indeed, given non-singular quadratic modules $(V,q_V)$ and $(W,q_W)$,
  $\text{Fix}(0\orthosum W)$ is just the set of automorphisms
  $\varphi\in O(V\orthosum W)$ which fix $0\orthosum W$ pointwise.
  If $\varphi$ is such a morphism, we must then have
  $\varphi (V\orthosum 0) = (0\orthosum W)^\bot = V\orthosum 0$;
  i.e. $\varphi$ must map the $V$-summand to itself,
  and letting $\varphi_V$ be the automorphism of $V$ thus defined,
  we find that $\varphi$ is the image of $\varphi_V$ under the map
  $O(V)\to O(V\orthosum W)$.
  
  To verify condition H1, let $(V,q_V)$ and $(W,q_W)$ be non-singular
  quadratic modules, and let $f,g\colon (V,q_V)\to (W,q_W)$
  be two maps between them.
  We note that $f$ and $g$ map $(V,q_V)$ isometrically onto
  non-singular quadratic submodules $f(V)$ and $g(V)\subseteq W$ respectively
  (cf. Proposition~\ref{prop:non-sing-summand}).
  For this, it suffices to note that $f$ and $g$
  are injective by Corollary~\ref{cor:injective}.

  Hence Proposition~\ref{prop:non-sing-summand} gives that
  $$
  (W,q_W) = (f(V),q_W)\orthosum (f(V)^\bot,q_W)
  \cong (V,q_W)\orthosum (f(V)^\bot,q_W)
  $$
  and similarly for $g(V)$,
  so
  $$
  (V,q_W)\orthosum (f(V)^\bot,q_W)
  \cong
  (V,q_W)\orthosum (g(V)^\bot,q_W).
  $$
  By the cancellation theorem, we thus have an isometry
  $h\colon (f(V)^\bot,q_W)\to (g(V)^\bot,q_W)$.
  But then $f = (fg^{-1}\orthosum h)\circ g$
  and $fg^{-1}\orthosum h\in O(q_W)$ as desired.
\end{proof}
\subsection{The space of destabilizations}
In this subsection, we identify the space of destabilizations
$|W_n(0,\E^1_A )_\bullet |$. 
First recall the following constructions~\cite{RWW17}:
\begin{defn}
  Let $(\mathcal C ,\oplus ,0)$ be a monoidal category
  in which the unit object $0$ is initial.

  Let $A$ and $X$ be objects in $\mathcal C$.
  For each $n\geq 1$, define
  \begin{enumerate}[label=(\roman*)]
  \item a semi-simplicial set
    $W_n(A,X)_\bullet$ having
    $$
    W_n(A,X)_p  = \text{Hom}_{\mathcal C}(X^{\oplus (p+1)},A\oplus X^{\oplus n})
    $$
    for each $p\geq 0$,
    with face maps
    $$
    \begin{tikzcd}[row sep =tiny]
      \text{Hom}_{\mathcal C}(X^{\oplus (p+1)},A\oplus X^{\oplus n})
      \arrow[r,"d_i"]
      & \text{Hom}_{\mathcal C}(X^{\oplus p},A\oplus X^{\oplus n}) \\
      f \arrow[r,mapsto]
      & f\circ (X^{\oplus i}\oplus (0\to X)\oplus X^{\oplus p-i})
    \end{tikzcd}
    $$
    for each $0\leq i\leq p$; and
  \item a simplicial complex $S_n(A,X)$ with set of vertices $W_n(A,X)_0$,
    and such that vertices $f_0,\dots ,f_p\in W_n(A,X)_0$ spans
    a simplex of $S_n(A,X)$ if and only if there is a simplex $\sigma\in W_n(A,X)_p$
    whose vertices are exactly $\lbrace f_0,\dots ,f_p\rbrace$.
  \end{enumerate}
\end{defn}
Intuitively, $|W_n(A,X)_\bullet |$ parametrizes
ways of cutting out copies of sums of $X$ from $A\oplus X^{\oplus n}$,
which is why~\cite{RWW17} refer to it as the $n$-th 
\emph{space of destabilizations} of $A$ by $X$.

We now identify the objects $W_n(0, \E^1_A )_\bullet$
and $S_n(0,\E^1_A )$.
\begin{defn}
  Let $(V,q)$ be a quadratic space over a commutative ring $R$.
  The \emph{ordered Stiefel space}
  $\vec X(q)_\bullet$ is the semi-simplicial set having
  $$
  \vec X (q)_p
  =
  \lbrace
  (v_0,\dots ,v_p)\in V^{p+1}\mid
  q(v_i)=1\text{ for each } i
  \text{ and }
  B_q(v_i,v_j) = 0 \text{ for } i\neq j
  \rbrace
  $$
  for each $p\geq 0$,
  with face maps
  $$
  \begin{tikzcd}[row sep =tiny]
    \vec X (q)_p
    \arrow[r,"d_i"]
    & \vec X (q)_{p-1} \\
    (v_0,\dots ,v_p) \arrow[r,mapsto]
    & (v_0,\dots ,\widehat{v_i},\dots ,v_p)
  \end{tikzcd}
  $$
  for each $0\leq i\leq p$.

  An element $(v_0,\dots ,v_p)\in \vec X(q)_p$ is referred to
  as an \emph{ordered frame} in $(V,q)$.
\end{defn}
\begin{prop}
  Let $R$ be a commutative ring
  and consider the symmetric monoidal category
  $(\Quad R ,\orthosum ,0)$.
  For any object $(V,q)\in \Quad R$ and each $n\geq 1$,
  there is an isomorphism of semi-simplicial sets
  $$
  W_n((V,q),\E^1_R)_\bullet\cong \vec X((V,q)\orthosum \E^n_R)_\bullet
  $$
  given on $p$-simplices by
  \begin{equation}
    \label{eq:destab-ident}
    f\mapsto (f(e_1),\dots ,f(e_{p+1}))
  \end{equation}
  % \begin{equation}
  %   \label{eq:destab-ident}
  %   \begin{tikzcd}[row sep = tiny]
  %     W_n((V,q),\E^1_R)_p\arrow[r] &\vec X((V,q)\orthosum \E^n_A)_p\\
  %     f\arrow[r,mapsto]hdh
  %     & (f(e_1),\dots ,f(e_{p+1}))
  %   \end{tikzcd}
  % \end{equation}
  for $0\leq i\leq p$,
  where $e_i$ denotes the unit vector in $\E^{p+1}_R$
  having $i$-th coordinate equal to $1$ and all other coordinates zero.
\end{prop}
\begin{proof}
  Note that $(f(e_1),\dots ,f(e_{p+1}))$ is indeed an ordered frame
  by Proposition~\ref{prop:preserve-bilinear},
  so the map is well-defined.
  To see that it is an isomorphism of semi-simplicial sets,
  note that for each $p$, the map~\eqref{eq:destab-ident} has an inverse
  $\vec X((V,q)\orthosum \E^n_A)_p \to W_n((V,q),\E^1_R)_p$
  given by sending $(v_0,\dots ,v_p)$ to the uniquely-determined $R$-linear
  map $f\colon R^{p+1}\to V\oplus R^n$ which maps $e_i$ to $v_{i-1}$.
  This map is form-preserving, since for each $\sum_{i=1}^{p+1}r_ie_i$
  we have
  \begin{align*}
    (q\orthosum n\langle 1\rangle )
    \left(
    f\left( \sum_{i=1}^{p+1}r_ie_i\right)
    \right)
    &= (q\orthosum n\langle 1\rangle )
      \left(
      \sum_{i=1}^{p+1}r_iv_{i-1}
      \right) \\
    &= \sum_{i=1}^{p+1}(q\orthosum n\langle 1\rangle )(r_iv_{i-1})
      + \sum_{1\leq i < j\leq p+1}
      r_ir_jB_{q\orthosum n\langle 1\rangle }(v_{i-1},v_{j-1})\\
    &= r_1^2 + \dots + r_{p+1}^2 = (p+1)\langle 1\rangle
      \left(\sum_{i=1}^{p+1}r_ie_i\right).
  \end{align*}
\end{proof}
\begin{cor}
  There is an isomorphism of simplicial complexes
  $$
  S_n((V,q),\E^1_R)\cong X(q\orthosum n\langle 1\rangle )
  $$
  given by $f\mapsto f(e_1) = f(1)$.
\end{cor}
\subsection{Homology stability}
As mentioned at the beginning of this section, the main theorem
of Randal-Williams--Wahl~\cite{RWW17} requires knowing that the
space of destabilizations
$|W_n(A,X)_\bullet |$ is highly-connected.
Under certain conditions, Randal-Williams and Wahl
show that connectivity estimates for $|S_n(A,X) |$ imply
connectivity estimates for $|W_n(A,X)_\bullet |$.
Specificially,
Prop 2.9 and Thm 2.10 in~\cite{RWW17} imply that if
$(\mathcal C,\oplus ,0)$ is a symmetric monoidal homogeneous category
which is \emph{locally standard} at a pair of objects $(A,X)$,
then for positive integers $a,k\geq 1$, the
space of destabilizations $|W_n(A,X)_\bullet |$ is $\left(\frac{n-a} k\right)$-connected
for all $n\geq 0$
if and only if
$|S_n(A,X) |$ is $\left(\frac{n-a} k\right)$-connected
for all $n\geq 0$.
Here being locally standard at $(A,X)$ means that
\begin{itemize}
\item[\textbf{LS1}] the morphisms $(0\to A)\oplus X\oplus (0\to X)$
  and $(0\to A\oplus X)\oplus X$ are distinct in $\text{Hom}_{\mathcal C}(X,A\oplus X^{\oplus 2})$,
  and
\item[\textbf{LS2}] for all $n\geq 1$, the map
  $$
  \begin{tikzcd}[row sep = tiny]
    \text{Hom}_{\mathcal C}(X,A\oplus X^{\oplus n-1})
    \arrow[r]
    & \text{Hom}_{\mathcal C}(X,A\oplus X^{\oplus n})\\
    f\arrow[r,mapsto]
    & f\oplus (0\to X)
  \end{tikzcd}
  $$
  is injective.
\end{itemize}
Note that
\begin{prop}
  If $R$ is a semi-local ring with $2\in R^\times$, then
  $\Quad R$ is locally standard at $(V,W)$ for
  any pair of quadratic modules $V$ and $W$ as long as $V\neq 0$.  
\end{prop}
\begin{proof}
  To see LS1, simply note that the two morphisms are given
  by $v\mapsto (0, v, 0)$ and $v\mapsto (0,0,v)$ respectively.
  Condition LS2 is clear from the definition of $\orthosum$.
\end{proof}
% We thus have:
% \begin{thm}
%   Let $R$ be a semi-local ring with $2\in R^\times$.
%   Suppose that there are integers $a\geq 0$ and $k\geq 2$ so that
%   for each $n$, the Stiefel complex $X(n\langle 1\rangle )$
%   is $\left(\frac{n-a-2}k\right)$-connected.
%   Then
%   $$
%   H_i(O_n(R);\Z )
%   \to
%   H_i(O_{n+1}(R);\Z )
%   $$
%   is a surjection for
%   $
%   i\leq \frac{n-a} k
%   $
%   and an isomorphism for
%   $
%   i\leq \frac{n-a-1} k.
%   $
% \end{thm} I think this should be changed to a remark, maybe in the
% introduction
Thus we have:
\begin{thm}\label{thm:theorema}
  Let $A$ be a valuation ring with residue field $\k$,
  and assume that $2\in A^\times$.
  Consider the map
  \begin{equation}
    \label{eq:stabilization}
    H_i(O_n(A);\Z )
    \to
    H_i(O_{n+1}(A);\Z )
  \end{equation}
  induced by
  $$
  \begin{tikzcd}[row sep = tiny]
    O_n(A)\arrow[r] &O_{n+1}(A) \\
    f\arrow[r,mapsto]
    & f\orthosum\text{id}.
  \end{tikzcd}.
  $$
  Then~\eqref{eq:stabilization}
  is
  \begin{enumerate}[label=(\roman*)]
  \item a surjection for
    $
    i\leq \frac{n-m_A-1}{3}
    $
    and an isomorphism for
    $
    i\leq \frac{n-m_A-2}{3}
    $
    if $m_A <\infty$; and
  \item a surjection for
    $
    i\leq \frac{n-m_A}{2}
    $
    and an isomorphism for
    $
    i\leq \frac{n-m_A-1}{2}
    $
    if $m_A<\infty$ and $A$ is formally real;
  \item a surjection for
    $
    i\leq \frac{n-3-2P(\k )}{2P(\k ) + 1}
    $
    and an isomorphism for
    $
    i\leq \frac{n-4-2P(\k )}{2P(\k ) + 1}
    $
    if $A$ is henselian and $P(\k )<\infty$; and
  \item a surjection for
    $
    i\leq \frac{n-2-P(\k )}{P(\k ) + 1}
    $
    and an isomorphism for
    $
    i\leq \frac{n-3-P(\k )}{P(\k ) + 1}
    $
    if $A$ is henselian, $\k$ is formally real, and $P(\k )<\infty$.
  \end{enumerate}
  Further, if $K$ denotes the quotient field of $A$,
  then~\eqref{eq:stabilization} is
  \begin{enumerate}[label=(\roman*)]
    \setcounter{enumi}{4}
  \item a surjection for
    $
    i\leq \frac{n-m_K-1}{3}
    $
    and an isomorphism for
    $
    i\leq \frac{n-m_K-2}{3}
    $
    if $m_K<\infty$;
  \item a surjection for
    $
    i\leq \frac{n-m_K}{2}
    $
    and an isomorphism for
    $
    i\leq \frac{n-m_K-1}{2}
    $
    if $K$ is formally real and $m_K<\infty$;
  \item a surjection for
    $
    i\leq \frac{n-3-2P(K )}{2P(K ) + 1}
    $
    and an isomorphism for
    $
    i\leq \frac{n-4-2P(K )}{2P(K ) + 1}
    $
    if $P(K)<\infty$; and
  \item a surjection for
    $
    i\leq \frac{n-2-P(K )}{P(K ) + 1}
    $
    and an isomorphism for
    $
    i\leq \frac{n-3-P(K )}{P(K ) + 1}
    $
    if $K$ is formally real and $P(K)<\infty$.
  \end{enumerate}
\end{thm}
\begin{proof}
  This follows from in~\cite[Thm 3.1]{RWW17}
  and Corollary~\ref{cor:cnt-ranges} above.
\end{proof}
\subsection{Stability with twisted coefficients}
The framework of~\cite{RWW17} also gives us homology stability of $O_n(A)$
with certain twisted coefficients, which we now describe.

Recall that a \emph{coefficient system} for the sequence
$\lbrace O_n(A)\rbrace_1^\infty$
is a functor $F\colon \mathcal E\to\mathcal A$,
where $\mathcal A$ is an abelian category
and $\mathcal E\subseteq \Quad A$ denotes the full subcategory
spanned by the Euclidean spaces $\E^n_A$.
In particular, functoriality endows $F_n = F(\E^n_A)$
with the structure of a $\Z\lbrack O_n(A)\rbrack$-module,
and the maps $\E^n_A\orthosum (0\to\E^1_A)\colon \E^n_A\to\E^{n+1}_A$
induce maps
\begin{equation*}
  H_i(O_n(A);F_n )
  \to
  H_i(O_{n+1}(A);F_{n+1} ).
\end{equation*}
Homology with twisted coefficients is a powerful tool for computing
the homology groups in a \emph{family of representations}
$\lbrace F_n\rbrace_1^\infty$.
\begin{thm}\label{thm:theorema-abelian}
  Let $A$ be a valuation ring with residue field $\k$,
  and assume that $2\in A^\times$.
  Put $O_\infty (A) = \varinjlim O_n(A)$,
  and suppose $M$ is an $O_\infty (A)$-module
  on which the commutator subgroup
  $\Omega_\infty (A)\leq O_\infty (A)$
  acts trivially.
  For each $n\geq 0$, we consider $M$
  as an $O_n(A)$-module via restriction of scalars
  along the canonical map $O_n(A)\to O_\infty (A)$.
  Consider the map
  \begin{equation}
    \label{eq:stabilization2}
    H_i(O_n(A);M )
    \to
    H_i(O_{n+1}(A);M )
  \end{equation}
  induced by
  $$
  \begin{tikzcd}[row sep = tiny]
    O_n(A)\arrow[r] &O_{n+1}(A) \\
    f\arrow[r,mapsto]
    & f\orthosum\text{id}
  \end{tikzcd}.
  $$
  Then~\eqref{eq:stabilization2}
  is
  \begin{enumerate}[label=(\roman*)]
  \item a surjection for
    $
    i\leq \frac{n-m_A-2}{3}
    $
    and an isomorphism for
    $
    i\leq \frac{n-m_A-4}{3}
    $
    if $m_A<\infty$;
  \item a surjection for
    $
    i\leq \frac{n-2P(\k ) - 2P(\k ) - 2}{2P(\k ) + 1}
    $
    and an isomorphism for
    $
    i\leq\frac{n-2P(\k ) - 2P(\k ) - 4}{2P(\k ) + 1}
    $
    if $A$ is henselian and $P(\k )<\infty$; and
  \item a surjection for
    $$
    \textstyle{i\leq \frac{n-P(\k )-\max\lbrace 3, P(\k ) + 1\rbrace}
      {\max\lbrace 3, P(\k ) + 1\rbrace}}
    $$
    and an isomorphism for
    $$
    \textstyle{i\leq \frac{n-P(\k )-2-\max\lbrace 3, P(\k ) + 1\rbrace}
      {\max\lbrace 3, P(\k ) + 1\rbrace}}
    $$
    if $\k$ is formally real, $A$ is henselian, and $P(\k)<\infty$.
  \end{enumerate}
  Further, if $K$ denotes the quotient field of $A$,
  then~\eqref{eq:stabilization2} is
  \begin{enumerate}[label=(\roman*)]
    \setcounter{enumi}{3}
  \item a surjection for
    $
    i\leq \frac{n-m_K-2}{3}
    $
    and an isomorphism for
    $
    i\leq \frac{n-m_K-4}{3}
    $
    if $m_K<\infty$;
  \item a surjection for
    $
    i\leq \frac{n-2P(K ) - 2P(K ) - 2}{2P(K ) + 1}
    $
    and an isomorphism for
    $
    i\leq\frac{n-2P(K ) - 2P(K ) - 4}{2P(K ) + 1}
    $
    if $P(K)<\infty$; and
  \item a surjection for
    $$
    \textstyle{i\leq \frac{n-P(K )-\max\lbrace 3, P(K ) + 1\rbrace}
      {\max\lbrace 3, P(K ) + 1\rbrace}}
    $$
    and an isomorphism for
    $$
    \textstyle{i\leq \frac{n-P(K )-2-\max\lbrace 3, P(K ) + 1\rbrace}
      {\max\lbrace 3, P(K ) + 1\rbrace}}
    $$
    if $K$ is formally real and $P(K)<\infty$.
  \end{enumerate}
\end{thm}
\begin{exmp}
  For each $n\geq 0$, let $\Omega_n(A)\leq O_n(A)$
  denote the commutator subgroup.
  Then the map $H_i(\Omega_n(A);\Z )\to H_i(\Omega_{n+1}(A);\Z )$
  is an epimorphism (resp. isomorphism) for the ranges given
  in the previous theorem (see~\cite[Sec 3.1]{RWW17}).
\end{exmp}
We recall what it means for a coefficient system to have \emph{degree $r$
  at $N\in\Z$}.
The unique map $0\to \E^1_A$ induces a natural transformation
$\sigma\colon \text{id}_{\mathcal E}\to {-}\oplus \E^1_A$.
For a coefficient system $F\colon\mathcal E\to\text{Mod}_\Z$, we put
\begin{align*}
  \operatorname{ker} F = \operatorname{ker} \sigma_F\qquad\text{and}\qquad
  \operatorname{coker} F = \operatorname{coker} \sigma_F,
\end{align*}
where $\sigma_F = F(\sigma)\colon
F\to F\circ ({-}\oplus \E^1_A)$.
Then $F$ is said to have degree $r<0$ at $N\in\Z$
if $F(\E^n_A) = 0$ for all $n\geq N$, and recursively it is
said to have degree $r\geq 0$
at $N\in\Z$ if
\begin{enumerate}[label=(\roman*)]
\item The kernel $\operatorname{ker} F$ has degree $-1$ at $N$; and
\item The cokernel $\operatorname{coker} F$ has degree $r-1$ at $N-1$.
\end{enumerate}
Note that if $F$ and $G$ are coefficient systems of degree $r$ and $s$
at $N$, then $F\oplus G$ has degree $\max\lbrace r,s\rbrace$ at $N$.
\begin{thm}\label{thm:poly-stab}
  Let $A$ be a valuation ring with residue field $\k$,
  and assume that $2\in A^\times$.
  Let $F\colon \mathcal E\to\text{Mod}_\Z$ be a coefficient system of degree
  $r$ at $N\in\Z$.
  Consider the map
  \begin{equation}
    \label{eq:stabilization3}
    H_i(O_n(A);F_n )
    \to
    H_i(O_{n+1}(A);F_{n+1} )
  \end{equation}
  induced by
  $$
  \begin{tikzcd}[row sep = tiny]
    O_n(A)\arrow[r] &O_{n+1}(A) \\
    f\arrow[r,mapsto]
    & f\orthosum\text{id}
  \end{tikzcd}.
  $$
  Then~\eqref{eq:stabilization3}
  is
  \begin{enumerate}[label=(\roman*)]
  \item a surjection for
    $
    i\leq \frac{n-m_A-1}{3} -r 
    $
    and an isomorphism for
    $
    i\leq \frac{n-m_A-1}{3} -r -1
    $
    if $m_A<\infty$;
  \item a surjection for
    $
    i\leq \frac{n-m_A}{2} -r 
    $
    and an isomorphism for
    $
    i\leq \frac{n-m_A}{2} - r - 1
    $
    if $\k$ is formally real and $m_A<\infty$;
  \item a surjection for
    $
    i\leq \frac{n-3-2P(\k )}{2P(\k ) + 1} - r
    $
    and an isomorphism for
    $
    i\leq \frac{n-3-2P(\k )}{2P(\k ) + 1} - r -1
    $
    if $A$ is henselian and $P(\k)<\infty$; and
  \item a surjection for
    $
    i\leq \frac{n-2-P(\k )}{P(\k ) + 1} - r
    $
    and an isomorphism for
    $
    i\leq \frac{n-2-P(\k )}{P(\k ) + 1} - r - 1
    $
    if $A$ is henselian, $\k$ is formally real and $P(\k)<\infty$.
  \end{enumerate}
  Further, if $K$ denotes the quotient field of $A$,
  then~\eqref{eq:stabilization3} is
  \begin{enumerate}[label=(\roman*)]
    \setcounter{enumi}{4}
      \item a surjection for
    $
    i\leq \frac{n-m_K-1}{3} -r 
    $
    and an isomorphism for
    $
    i\leq \frac{n-m_K-1}{3} -r -1
    $
    if $m_K<\infty$;
  \item a surjection for
    $
    i\leq \frac{n-m_K}{2} -r 
    $
    and an isomorphism for
    $
    i\leq \frac{n-m_K}{2} - r - 1
    $
    if $K$ is formally real and $m_K<\infty$;
  \item a surjection for
    $
    i\leq \frac{n-3-2P(K )}{2P(K ) + 1} - r
    $
    and an isomorphism for
    $
    i\leq \frac{n-3-2P(K )}{2P(K ) + 1} - r -1
    $
    if $P(K)<\infty$; and
  \item a surjection for
    $
    i\leq \frac{n-2-P(K )}{P(K ) + 1} - r
    $
    and an isomorphism for
    $
    i\leq \frac{n-2-P(K )}{P(K ) + 1} - r - 1
    $
    if $K$ is formally real and $P(K)<\infty$.
 \end{enumerate}
\end{thm}
\subsection{Homology stability of $O_n(\Z )$}
In~\cite{vogtmann82}, Vogtmann states that the Euclidean orthogonal groups
$O_n(R)$ also display homology stability when $R$
 is the ring of integers
in a totally real number field $K$.\footnote{
Recall that a field $K$ is \emph{totally real} if for every
embedding $\iota\colon K\hookrightarrow\C$, one has $\iota (K)\subseteq\R$.}
In particular, this includes the case $R=\Z$.

Confusingly, Vogtmann seems to treat part of this integral case
together with the field case, 
thus using Witt cancellation and hyperplane
reflections to show that $O_n(R)$
acts transitively on $p$-simplices in the Stiefel complex $X(\E^n_R)$.
However, hyperplane reflections do not exist over $R$ 
(as $2\not\in R^\times$), and Witt cancellation does not hold in general.
Instead, transitivity -- or, as it corresponds to in the \cite{RWW17} setup, 
\emph{local} homogeneity of $\Quad R$ -- can be deduced from
 \cite[Lem 1.3]{vogtmann82}, which states that the only unit vectors in
$\E^n_R$ are of the form $\pm e_i$, where $e_1,\dots ,e_n$ denote
the standard basis of $R^n$.
Clearly $O_n(R)$ then acts transitively on $p$-frames.

However, the fact that the unit vectors in $\E^n_R$ admit such a
simple description also means that $O_n(R)$ does not really contain
any ring-theoretic information. In fact,
\begin{prop}
  Let $R$ be the ring of integers in a totally real number field,
  and denote by $\mathrm{Aut}(X(\E^n_R))$ the group of automorphisms
  of the simplicial complex $X(\E^n_R)$.
  The homomorphism
  $$
  O_n(R)\to\mathrm{Aut}(X(\E^n_R)),
  $$
  coming from the action of $O_n(R)$ on $X(\E^n_R)$,
  is an isomorphism of groups.
\end{prop}
\begin{proof}
  Note that an element $\varphi\in O_n(R)$
  is determined by the images of the standard basis $\varphi (e_i)$,
  $i=1,\dots ,n$, so injectivity is clear.

  For surjectivity, let $f\in \mathrm{Aut}(X(\E^n_R))$.
  If $\sigma = \lbrace e_1,\dots ,e_n\rbrace$ is the simplex
  corresponding to the standard basis,
  then $f(\sigma ) = \lbrace f(e_1),\dots ,f(e_n)\rbrace$
  is again an $n$-frame, so there is $\varphi\in O_n(R)$
  with $\varphi (e_i) = f(\lbrace e_i\rbrace )$.
  But any automorphism of $X(\E^n_R)$ is uniquely determined
  by the images of the vertices corresponding to $e_1,\dots ,e_n$,
  as the vertex corresponding to $-e_i$ can be characterized as
  the unique vertex in $X(\E^n_R)$ (other than $e_i$ itself)
  which does not share an edge with $e_i$.
  Hence $\varphi$ must be the desired lift of $f$.
\end{proof}
In other words, ring theory disappears and we are left with
combinatorics; the groups $O_n(R)$ are just automorphism groups
for the simplicial complexes in the sequence of simplicial spheres
$$
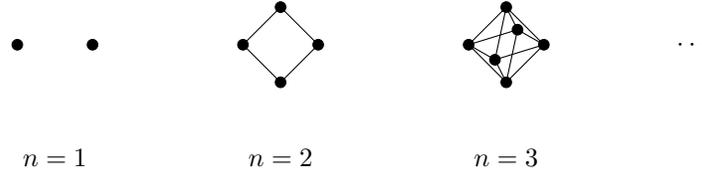
\begin{tikzpicture}
  \filldraw[black] (0,0) circle (2pt) node[anchor=south west]{};
  \filldraw[black] (1,0) circle (2pt) node[anchor=south west]{};

  \filldraw[black] (3,0) circle (2pt) node[anchor=north east]{};
  \filldraw[black] (3.5,0.5) circle (2pt) node[anchor=south west]{};
  \filldraw[black] (3.5,-0.5) circle (2pt) node[anchor=north]{};
  \filldraw[black] (4,0) circle (2pt) node[anchor=north west]{};

  \draw (3,0)--(3.5,0.5);
  \draw (3,0)--(3.5,-0.5);
  \draw (4,0)--(3.5,0.5);
  \draw (4,0)--(3.5,-0.5);

  \filldraw[black] (6,0) circle (2pt) node[anchor=north west]{};
  \filldraw[black] (6.65,0.2) circle (2pt) node[anchor=north west]{};
  \filldraw[black] (6.35,-0.2) circle (2pt) node[anchor=north west]{};
  \filldraw[black] (7,0) circle (2pt) node[anchor=north west]{};
  \filldraw[black] (6.5,0.5) circle (2pt) node[anchor=north west]{};
  \filldraw[black] (6.5,-0.5) circle (2pt) node[anchor=north west]{};

  \draw (6,0)--(6.65,0.2);
  \draw (6,0)--(6.35,-0.2);
  \draw (7,0)--(6.65,0.2);
  \draw (7,0)--(6.35,-0.2);

  \draw (6.5,-0.5)--(6,0);
  \draw (6.5,-0.5)--(7,0);
  \draw (6.5,-0.5)--(6.65,0.2);
  \draw (6.5,-0.5)--(6.35,-0.2);

  \draw (6.5,0.5)--(6,0);
  \draw (6.5,0.5)--(7,0);
  \draw (6.5,0.5)--(6.65,0.2);
  \draw (6.5,0.5)--(6.35,-0.2);

  \node at (0.5,-1.5) {$n=1$};
  \node at (3.5,-1.5) {$n=2$};
  \node at (6.5,-1.5) {$n=3$};
  \node at (9,0) {$\cdots$};
\end{tikzpicture}
$$
\subsection{$K$-theoretic interpretation}
\label{subsec:k-thy}
As mentioned in the introduction, interest in the split-orthogonal
groups $O_{n,n}$ has been driven by its importance in hermitian
$K$-theory.
We recall the role of $BO_{\infty ,\infty}^+$ in $K$-theory here
and give an analogous interpretation of $BO_\infty^+$.

Let $R$ be a commutative ring.
\begin{prop}
  The inclusion of the full monoidal subcategory
  $\lbrace \H^{2n}\mid n\geq 0\rbrace\subseteq\Quad R$
  induces a homotopy equivalence
  on basepoint components of $K$-spaces.
  Hence
  $$
  K(\Quad R)\simeq K_0(\Quad R)\times BO_{\infty ,\infty }(R)^+,
  $$
  where $O_{\infty ,\infty }(R) = \varinjlim_n O_{n,n}(R)$.
\end{prop}
\begin{proof}
  For any projective $R$-module $V$, let $\H (V)$ denote the
  quadratic module $(V\oplus V^*,q_{\H })$
  where $q_{\H }(m,\phi ) = \phi (m)$.
  The quadratic module $\H (V)$ is called the \emph{hyperbolic module}
  associated to $M$,
  justified by the fact that $\H (R^n)\cong \H^{2n}$.
  More generally, one finds by an elementary argument that
  $(V,q)\orthosum (V,-q)\cong \H (V)$
  for any quadratic module $(V,q)$.
  If $(V,q)$ is a quadratic module, then since $V$
  is finitely-generated projective, there is an
  isomorphism of modules $R^n\cong V\oplus U$ for
  some $n$ and $U$.
  Thus we find $(V,q)\orthosum (V,-q)\orthosum\H (U)
  \cong \H (V)\orthosum\H (U)
  \cong \H (V\oplus U)
  \cong\H (R^n)\cong\H^{2n}$.
  This shows that
  $\lbrace \H^{2n}\mid n\geq 0\rbrace$
  is cofinal in $\Quad R$,
  so the inclusion of this monoidal subcategory
  induces a homotopy equivalence
  on basepoint components of $K$-spaces.
  Since $\lbrace \H^{2n}\mid n\geq 0\rbrace$
  is generated by $\H^2$,
  a standard group completion argument (see~\cite{rw13} or~\cite{nikolaus17})
  shows that
  $$
  K(\lbrace \H^{2n}\mid n\geq 0\rbrace )
  \simeq\hocolim \left(
    \mathcal H\xrightarrow{{-}\orthosum\H^2}
    \mathcal H\xrightarrow{{-}\orthosum\H^2}\dots
  \right)^+,
  $$
  where $\mathcal H$ is the classifying space of the groupoid core
  of $\lbrace \H^{2n}\mid n\geq 0\rbrace$,
  i.e. $\mathcal H = \coprod_{n\geq 0} BO_{n,n}(R)$.
  But here
  $$
  \hocolim \left(
    \mathcal H\xrightarrow{{-}\orthosum\H^2}
    \mathcal H\xrightarrow{{-}\orthosum\H^2}\dots
  \right)\simeq\Z\times BO_{\infty ,\infty }(R),
  $$
  and the conclusion follows.
\end{proof}
Here $K(\Quad R )$ is the hermitian $K$-theory space of Karoubi
and others, also denoted $\mathscr L(R)$.
By analogy, the space $BO_\infty(R)^+$ is related to
``positive-definite quadratic $K$-theory''.
\begin{prop}
  Let $\operatorname{Quad}^+( R)\subseteq\Quad R$ be the full monoidal
  subcategory spanned by non-singular submodules of Euclidean spaces.
  Then
  $$
  K\left(
    \operatorname{Quad}^+( R)
  \right)\simeq
  K_0\left(
    \operatorname{Quad}^+( R)
  \right)
  \times BO_\infty(R)^+.
  $$
\end{prop}
Note that if $R$ is a field, or more generally a local ring,
a non-singular quadratic module $(M,q)$ lies in $\operatorname{Quad}^+(R)$
if and only if $q(x)$ is a sum of squares for each $x\in M$.

\appendix
\section{Arithmetic invariants of local rings}
\label{sec:arithm-invars}
\subsection{Basic facts and definitions}
The unnamed invariant $m_A$ is related to other ring invariants
on which there is a rich and well-developed theory
(see Lam's textbook~\cite{lam-textbook} for an overview).
For each $n\in\N$, let
$$
S_n(A) = \lbrace
x_1^2 + \dots + x_n^2\mid x_1,\dots ,x_n\in A
\rbrace\subseteq A.
$$
Note that $S_n(A)$ is exactly the set of elements of $A$
that are represented by Euclidean $n$-space $\E_A^n$.
We define the invariants
\begin{align*}
  P(A) &= \inf\lbrace
         p\in\N\mid \bigcup_{n=1}^\infty S_n(A)
         = S_p(A)
         \rbrace \tag{\emph{Pythagoras number}},\\
  s(A) &= \inf\lbrace
         s\in\N\mid -1\in S_s(A)
         \rbrace \tag{\emph{Stufe}},\\
  u(A) &= \sup\lbrace
         \dim_{\k} \redu V
         \mid\text{$(V,q)\in\Quad A$ is isotropic}
         \rbrace.
\end{align*}
The invariant $u(A)$ is sometimes called the \emph{$u$-invariant}
of $A$.

The letter $u$ comes from an alternative characterization, which
we give now.
\begin{defn}
  A non-singular quadratic module $(V,q)$ over a commutative ring $R$
  is said to be \emph{universal} if it represents every unit of $R$.
\end{defn}
\begin{prop}
  Let $A$ be a local ring with residue field $\k$,
  and assume that $2\in A^\times$.
  Then $u(A)$ equals
  the smallest number $u$
  such that every non-singular quadratic module $(V,q)$
  with $\dim_\k\redu V\geq u$ is universal.
\end{prop}
\begin{proof}
  Let $\tilde u\in\N\cup\lbrace\infty\rbrace$ be the smallest number with
  respect to the property that
  every non-singular quadratic module $(V,q)$ with $\dim_\k\redu V\geq\tilde u$
  is universal. We must show that $\tilde u = u(A)$, as defined previously.
  
  Let $(V,q)$ be a non-singular quadratic module with $\dim_\k\redu V\geq u(A)$.
  We claim that $(V,q)$ is universal.
  Indeed, let $a\in A^\times$.
  Then $\dim_{\k}\redu V\oplus\k > u(A)$, so the non-singular quadratic module
  $(V\oplus A, q\orthosum \langle -a\rangle )$ is isotropic.
  By the representation theorem, we thus have that $q$ represents $a$,
  so $V$ is universal.
  Hence $u(A)\geq\tilde u$

  For the reverse inequality, suppose $(V,q)$ has $\dim_\k\redu V\geq\tilde u$.
  We must show that $\redu V$ is isotropic.
  By the diagonalization theorem, we may assume that
  $q\cong \langle a_1,\dots ,a_n\rangle$ with $a_1,\dots ,a_n\in A^\times$.
  Then $\langle a_1,\dots ,a_{n-1}\rangle$ is universal, hence
  represents $-a_n$, and the representation theorem gives that
  $\langle a_1,\dots ,a_n\rangle$ is isotropic as desired.
\end{proof}
It follows immediately from the proposition that when $2\in A^\times$,
one has
$$
m_A\leq u(A)\quad\text{and}\quad s(A)\leq u(A).
$$
Also, if $-1 = \sum_{j=1}^sa_j^2$,
then for \emph{any} element $x\in A$ we find
$$
x = \left(
  \frac{x+1} 2
\right)^2
+ \sum_{j=1}^s
\left(
  \frac{a_j(x-1)} 2
\right)^2,
$$
so $x$ is a sum of $s+1$ squares.
In particular, this implies that
\begin{equation}\label{eq:pythagoras-stufe}
  P(A)\leq s(A)+1,
\end{equation}
still under the assumption that $2\in A^\times$.
Note also that if $s(A) < \infty$,
then $s(A)\leq P(A)$.

\begin{prop}\label{prop:invariant-examples}
  \leavevmode
  \begin{enumerate}[label=(\roman*)]
  \item If $F$ is a field, then $F$ is Pythagorean if and only if $m_F$ = 1.
  \item $s(\F_q)\in\lbrace 1,2\rbrace$, 
    $m_{\F_q} = P(\F_q ) = u(\F_q ) = 2$ for $q=p^n$,
    $p\neq 2$.
  \item If $F$ is a local or global field, then $m_F\leq 4$.
  \end{enumerate}
\end{prop}
\begin{proof}
  (i) Recall that a Pythagorean field is a field whose Pythagoras
  number is one.
  Suppose $F$ is Pythagorean and 
  let $V$ be a non-singular nonzero subspace of $\E^n_F$.
  Pick an orthogonal basis $\lbrace v_1,\dots ,v_k\rbrace$
  of $V$, and let $a_j = q(v_j)$ for each $j$.
  Each $a_j$ is a sum of squares, so the assumption
  provides $b\in F$ with $b^2 = a_1$.
  Then $1/b\cdot v_1$ is a unit vector in $V$.
  Conversely, suppose $m_F = 1$ and let $a\in F$ be a sum of squares,
  say, $a = x_1^2 + \dots + x_n^2$. We may assume $a\neq 0$ since
  $0 = 0^2$. Then the span of $v = (x_1,\dots ,x_n)$ is a non-singular
  subspace of $\E^n_F$, so by assumption there is $b\in F$ with
  $1 = n\langle 1\rangle (bv) = b^2a$, implying that $a = (1/b)^2$.

  (ii) For the fact that every non-singular binary quadratic
  form over $\F_q$ is universal, see for instance~\cite[Prop II.3.4]{lam-textbook}.
  To show the remaining statements, it suffices by (i) and the
  inequalities preceding the proposition to see that $P(\F_q)> 1$.
  Non-squares exist in $\F_q$, as otherwise the function $\F_q\to\F_q$,
  $x\mapsto x^2$, would be surjective and thus also injective.\footnote{
    I thank Kasper Andersen for teaching me this elementary argument. 
    }
  But
  $1 = 1^2 = (-1)^2$ and $1\neq -1$ since $q$ is not a power of two.
  On the other hand, since $u (\F_q ) < \infty$, every element of $\F_q$ is a sum
  of squares; hence $P(\F_q ) > 1$.

  (iii) In the local field case, we invoke the classification of local fields.
  If $F$ is $\R$ or $\C$, then $F$ is pythagorean and we have
  $m_F = 1$ by (i).
  Otherwise $F$ is a local field with finite residue class field and
  $u(F)\leq 4$ by~\cite[Thm VI.2.12]{lam-textbook}, so $m_F\leq u(F)\leq 4$.
  The global field case follows from the local field case by the
  Hasse--Minkowski theorem.
\end{proof}
Various other bounds on $P(F)$, $s(F)$ and $u(F)$ can be found in
the literature on these invariants.
We highlight the following difficult result, a proof of which
is found in Scharlau's textbook~\cite{scharlau}:
\begin{thm}[Tsen--Lang]
  Let $F$ be a field of transcendence degree $d$ over an algebraically
  closed field $K$. Then $u(F)\leq 2^d$.
\end{thm}
\subsection{Invariants of henselian local rings}
Knowing that there are plenty of fields with finite arithmetic invariants,
we next consider how arithmetic properties of a local ring $A$ relate to
those of its residue field $\k$.
\begin{prop}\label{prop:arithm-invars-ineqs}
  Let $A$ be a local ring with residue field $\k$. Then
  \begin{align}
    m_\k &\leq m_A,\label{res-m-ineq} \\
    P(\k )&\leq P(A),\label{res-p-ineq} \\
    s(\k )&\leq s(A),\label{res-s-ineq}\\
    u(\k )&\leq u(A)\label{res-u-ineq}
  \end{align}
  with equality in~\eqref{res-m-ineq},~\eqref{res-s-ineq},
  and~\eqref{res-u-ineq}
  if $A$ is henselian.
\end{prop}
The proof will use two lemmata.
The first is essentially \cite[Lem V.1.4]{baeza78}:
\begin{lem}\label{lem:hensel-isotropy}
  Let $A$ be a henselian local ring and
  let $(V,q)$ be a non-singular quadratic module over $A$.
  \begin{enumerate}[label=(\roman*)]
  \item If the reduction $(\redu V ,\redu q )$ is isotropic,
    then $(V,q)$ is isotropic.
  \item Let $a\in A^\times$. If $(\redu V,\redu q)$ represents
    $\overline a$, then $(V,q)$ represents $a$.
  \end{enumerate}
\end{lem}
\begin{proof}
  (i) Since $\redu V$ is isotropic and non-singular, we may pick
  a hyperbolic pair $\redu u,\redu v\in\redu V$ with representatives
  $u,v\in V$. Put $a = q(u)$, $b=B_q(u,v)$ and $c = q(v)$ and
  consider the polynomial $f(X) = aX^2 + bX + c$.
  After reduction modulo $\mathfrak m$, we have $\bar f (X) = X$,
  so $\bar f(0) = 0$ and $\bar f '(X) = 1\neq 0$.
  Using that $A$ is henselian, we find a root $\lambda\in\mathfrak m$
  for the unreduced polynomial.
  But then
  $$
  q(\lambda u + v) = \lambda^2 q( u) + q(v) + \lambda B_q(u,v)
  = f(\lambda ) = 0.
  $$
  (ii) Follows from (i) and the representation theorem.
\end{proof}
\begin{lem}\label{lem:lift-nonsing-subsp}
  Let $A$ be a local ring with residue field $\k$.
  Let $(V,q)$ be a quadratic module over $A$
  and let $\mathbf U\subseteq\redu V$ be a
  non-singular subspace of the reduction.
  Then there is a quadratic subspace $U\subseteq V$ whose reduction
  $\redu U$ is isometric to $\mathbf U$.
\end{lem}
\begin{proof}
  Choose an orthogonal basis $\bar u_1,\dots ,\bar u_n$ of
  $\mathbf U\subseteq\redu V = V/\mathfrak m V$
  and pick representatives $u_1,\dots ,u_n\in V$.
  Let $U = Au_1 + \dots + Au_n\subseteq V$.
  We claim that $U$ is a free $A$-module and that $u_1,\dots ,u_n$
  is a basis.
  Indeed, suppose $\sum_{i=1}^na_iu_i=0$ with $a_1,\dots ,a_n\in A$.
  Applying $B_q({-},u_j)$ for $j=1,\dots ,n$, we get a
  system of $n$ equations
  $$
  \sum_{i=1}^na_iB_q(u_i,u_j)=0,\qquad j =1,\dots ,n.
  $$
  Note that
  $\redu{\det (B_q(u_i,u_j))_{i,j}} = \det (B_{\redu q}(\bar u_i,\bar u_j))
  \neq 0$
  by non-singularity of $\mathbf U$,
  so $\det (B_q(u_i,u_j))_{i,j}\in A^\times$
  and hence we must have $a_i = 0$ for each $i$
  by~\cite[Prop XIII.4.16]{lang}.
  Thus $U$ is free
  and the isometry $\redu U\cong \mathbf U$ is obvious.
  It follows from Propositions~\ref{prop:singular-reduc}
  and~\ref{prop:non-sing-summand} that $U$ is a submodule of $V$.
\end{proof}
\begin{proof}[Proof of Proposition~\ref{prop:arithm-invars-ineqs}]
  Inequalities~\eqref{res-p-ineq} and~\eqref{res-s-ineq} follow from the
  easy observation that if $f\colon R\to S$ is a surjection of rings,
  then $P(S)\leq P(R)$ and $s(S)\leq s(R)$.
  To see~\eqref{res-m-ineq}, suppose $m\geq m_A$.
  If $\mathbf V$ is an $m$-dimensional non-singular subspace of $\E^n_{\k}$ for some $n$,
  then Lemma~\ref{lem:lift-nonsing-subsp} says there is a non-singular
  subspace $V\subseteq\E^n_A$
  with $\redu V\cong \mathbf V$. But then by assumption $V$ contains a unit vector,
  and hence so does $\mathbf V$.
  To see~\eqref{res-u-ineq},
  let $(\mathbf V,q)$ be an arbitrary $d$-dimensional
  non-singular quadratic module over $\k$ with $d\geq u(A)$.
  Pick an orthogonal basis $\mathbf V\cong\langle \bar a_1,\dots ,\bar a_d\rangle$.
  Then $V = \langle a_1,\dots ,a_d\rangle$ is a $d$-dimensional non-singular
  quadratic module over $A$, hence is universal.
  It follows that $\mathbf V\cong\redu V$ is universal.
  
  Suppose now that $A$ is henselian.
  If $V$ is a non-singular quadratic subspace of $\E^n_A$
  for some $n$ having $\dim_{\k}\redu V\geq m_{\k}$,
  then the reduction $\redu V$ is a non-singular quadratic subspace of
  $\E^n_{\k}$. Since $\dim_{\k}\redu V \geq m_{\k}$,
  we have that $\redu V$ contains a unit vector, and
  Lemma~\ref{lem:hensel-isotropy} implies that $V$ contains a unit vector.
  Similar applications of Lemma~\ref{lem:hensel-isotropy}
  show that equalities hold in~\eqref{res-s-ineq} and~\eqref{res-u-ineq}.
\end{proof}
\begin{rmk}
  \begin{enumerate*}[mode=unboxed,itemjoin = \hspace{2cm plus 2cm},label=(\roman*)]
  \item It is far from true that equality holds in~\eqref{res-p-ineq}
    if $A$ is henselian.
    Indeed, one very quickly runs into $P(A) = \infty$
    if finiteness of $P(A)$ is not forced by finiteness of the Stufe.
    Recall that a field is called \emph{formally real} if its Stufe
    is infinite.
    Choi, Dai, Lam and Reznick~\cite{lam82} have shown that if $R$
    is a regular local ring with $\dim_{\text{Krull}}(R) \geq 3$
    whose residue field is formally real, then $P(R) = \infty$.
    In particular, if $\k$ is formally real
    then the complete local ring $A = \k \llbracket X_1,\dots ,X_d\rrbracket$
    has $P(A) = \infty$ for all $d\geq 3$. \\
  \item Some condition on the local ring $A$ is necessary for equality to hold
    in~\eqref{res-m-ineq},~\eqref{res-s-ineq}, and~\eqref{res-u-ineq}.
    Let $p$ be an odd prime.
    The ring $\Z_{(p)}$ has $s(\Z_{(p)}) = u(\Z_{(p)}) = \infty$, but
    ${\Z_{(p)}/ p\Z_{(p)}}\cong\F_p$ has $s(\F_p )\leq 2$ and $u(\F_p ) = 2$,
    which shows that a condition is necessary for equality to
    hold in~\eqref{res-s-ineq} and~\eqref{res-u-ineq}.
    Furthermore, $m_{\F_p} = 2$ whereas $m_{\Z_{(p)}} = 4$,
    showing that we cannot in general expect equality to hold in~\eqref{res-m-ineq}.
    To see that $m_{\Z_{(p)}}\geq 4$,
    pick $a,b\in\N$ such that $p\nmid 4^a(8b-1)$;
    for instance, one could take $a=0$ and $b=p$.
    Then $3\left\langle \frac{1}{4^a(8b-1)}\right\rangle$
    is a non-singular subspace of $\E^{12}_{\Z_{(p)}}$ since
    $\frac{1}{4^a(8b-1)}$ is a sum of four squares in $\Z_{(p)}$
    by Lagrange's four-square theorem.
    Indeed, the four-square theorem says there is
    $r_1,r_2,r_3,r_4\in\Z$ with $r_1^2+r_2^2+r_3^2+r_4^2 = 4^a(8b-1)$.
    But then
    $$
    \left(
      \frac{r_1}{4^a(8b-1)}
    \right)^2
    +
    \left(
      \frac{r_2}{4^a(8b-1)}
    \right)^2
    +
    \left(
      \frac{r_3}{4^a(8b-1)}
    \right)^2
    +
    \left(
      \frac{r_4}{4^a(8b-1)}
    \right)^2
    = \frac 1{4^a(8b-1)},
    $$
    and $\frac{r_j}{4^a(8b-1)}\in\Z_{(p)}$ for each $j$.
    However, $3\left\langle \frac{1}{4^a(8b-1)}\right\rangle$
    does not contain a unit vector.
    To see this, assume for contradiction that
    $(x,y,z)\in\Z_{(p)}^3\subseteq\Q^3$
    is a unit vector, i.e. that
    $$
    \frac{x^2+y^2+z^2}{4^a(8b-1)} = 1.
    $$
    Clearing the denominator, we find
    $$
    x^2 + y^2 + z^2 = 4^a(8b-1).
    $$
    But (the trivial direction of) Legendre's three-square theorem
    together with the Davenport--Cassels
    theorem say that $4^a(8b-1)$ cannot be expressed as a sum of three
    squares of rational numbers, contradiction.
    In fact it follows from Proposition~\ref{prop:arithm-invars-of-DVR}
    below that $m_{\Z_{(p)}}\leq m_\Q = 4$,
    and hence $m_{\Z_{(p)}} = 4$ as claimed.
  \end{enumerate*}
\end{rmk}
\subsection{Invariants of valuation rings}
Finally, we consider how the
arithmetic invariants of a valuation ring are also
related to those of its quotient field.
\begin{lem}\label{lem:DVR-isotropy}
  Let $A$ be a valuation ring.
  and let $(V,q)$ be a non-singular quadratic module over $A$.
  Let $K$ be the quotient field of $A$
  and denote by $(V_K,q_K)$ the quadratic space over $K$ with
  $V_K = K\otimes_A V$ and $q_K(\alpha\otimes x) = \alpha^2q(x)$ for
  $\alpha\in K$, $x\in V$.
  Then
  \begin{enumerate}[label=(\roman*)]
  \item $(V,q)$ is isotropic if and only if $(V_K,q_K)$ is.
  \item Then $(V,q)$ represents $a\in A^\times$
    if and only if  $(V_K,q_K)$ does.
  \end{enumerate}
\end{lem}
\begin{proof}
  Suppose $(V_K,q_K)$ is isotropic.
  We may assume that $V=A^n$.
  Let $0\neq x\in V$ with $q_K(x) = 0$,
  and let $\nu\colon K\to\Gamma\cup\lbrace \infty\rbrace$
  be a valuation for $A$.
  Write $x = (x_1,\dots ,x_n)$ and pick $x_i$ with
  $\nu (x_i)\leq \nu (x_j)$ for all $j$.
  Then $x_i^{-1}x\in A^n$ since $\nu (x_i^{-1}x_j) = \nu (x_j) - \nu (x_i)\geq 0$
  for each $j$. Also, $x_i^{-1}x$ is primitive since its
  $i$th coordinate is $x_i^{-1}x_i = 1$. But
  $$
  q(x_i^{-1}x) = q_K(x_i^{-1}x) = x_i^{-2}q_K(x) = 0.
  $$
  This proves the nontrivial direction of (i). Further, (ii)
  follows from (i) via the representation theorem.
\end{proof}
\begin{prop}\label{prop:arithm-invars-of-DVR}
  Let $A$ be a valuation ring with $2\in A^\times$,
  and denote the quotient field of $A$ by $K$.
  Then
    \begin{align}
    m_A &\leq m_K,\label{m-ineq} \\
    s(A )&= s(K),\label{s-ineq}\\
    \intertext{and}
    u(A )&\leq u(K)\label{u-ineq}
  \end{align}
\end{prop}
\bibliographystyle{amsalpha}
\bibliography{sources}
\end{document}